\documentclass[12pt]{article}

\usepackage{epsfig, amsfonts, amsmath, amsthm,amsbsy,subfigure}
\usepackage{color}
\def\ve{\varepsilon}
\newtheorem{theorem}{Theorem}
\newtheorem{corollary}{Corollary}
\newtheorem{remark}{Remark}

\begin{document}

\title{\bf A numerical algorithm to computationally solve the Hemker problem using Shishkin meshes}
\author{A. F. Hegarty 
\thanks{Department of Mathematics and Statistics, University of Limerick, Ireland.} \and E.\ O'Riordan \thanks{ School of Mathematical Sciences, Dublin City University, Dublin 9, Ireland.} }

\maketitle

{{\centerline{\it Dedicated to the memory of P. W. Hemker, who inspired this research work}}}
\begin{abstract}

 A numerical algorithm is presented to solve  a benchmark   problem proposed by Hemker\cite{hemker}. The algorithm incorporates asymptotic information into the design of appropriate piecewise-uniform Shishkin meshes. Moreover, different co-ordinate systems are utilized due to the different geometries and associated layer structures that are involved in this problem. Numerical results are presented to demonstrate the effectiveness of the proposed numerical algorithm. 
\bigskip

\noindent{\bf Keywords:} Singularly perturbed,  Shishkin mesh, Hemker problem.

\noindent {\bf AMS subject classifications:}    65N12, 65N15, 65N06.
\end{abstract}

\section{Introduction}

In \cite{hemker} Hemker proposed a model test problem in two space dimensions, defined on the unbounded domain $ \mathbb{R}^2 \setminus \{ x^2+y^2 \leq 1\} $, which is exterior to the unit circle. The problem involves the simple constant coefficient linear singularly perturbed convection-diffusion equation
\begin{equation}\label{Piet's-problem}
-\ve { \Delta }  u + u_x=0, \quad \hbox{for} \quad x^2+y^2 >1;  
\end{equation}
with the boundary conditions $u(x,y) =1$,  if $x^2+y^2 = 1$ and $ u(x,y) \rightarrow 0$ as $x^2+y^2 \rightarrow \infty$. 
An exponential boundary layer and characteristic interior layers appear in the solution of this problem.  In neighbourhoods of the two points $(0,\pm1)$, where the characteristic lines $y=C$  of the reduced problem 
($ v_x=0$) are tangential to the circle, there are transition regions, where the steep gradients in the solution migrate from the exponential boundary layers (located on the left side of the unit disk) to the characteristic internal layers  
which are emanating from the characteristic points $(0,\pm1)$ (see Figure \ref{Fig.0}). To design a numerical method, which produces stable and accurate approximations over the entire domain, for arbitrary  
small values of the perturbation parameter $\ve$, is seen in \cite{john1, bosco, havik, hemker, volker1}  as a reasonable challenge to the numerical analysis community. Here we concentrate on the layers that   appear in the vicinity of the circle by considering the problem restricted to a finite domain.  Hence, we do not address the potential merging of the two characteristic layers, which can occur at a sufficentlly large distance of $O(1/{\ve})$ \cite[pg.190]{{lagerstrom}} downwind (i.e.,  $x >> 1$)  of the circle.

In several publications (e.g.\cite{john1, gabriel}) the Hemker problem is  used to test the stability of numerical algorithms  designed for a wide class of convection-dominated convection-diffusion problems, as classical finite element methods produce spurious oscillations for this type of singularly perturbed problems. 
In \cite{tailor1}, a stable numerical method was constructed  on a quasi-uniform mesh for problem (\ref{Piet's-problem}), but there was a limited discussion of the accuracy of the numerical approximations.
In this paper, we guarantee parameter-uniform stability of the discrete  operators by using simple upwinded finite difference operators.
 Our main focus is on the design of an appropriate layer-adapted mesh, so that we guarantee that a significant proportion of the mesh  points lie within the layers.
In Theorem \ref{main}, several pointwise bounds on the solution of the continuous problem are established, from which the location and width of any layers are identified. 
 Asymptotic analysis \cite{eckhaus, hemker, ilin, lagerstrom, waechter} has also been used to determine the location and scale of all the layers that can appear in the solution of problem (\ref{Piet's-problem}).  

The numerical algorithm constructed in this paper is composed of several different Shishkin meshes \cite{gis1} defined across different co-ordinate systems aligned to  three overlapping subdomains. As we lack sufficient theoretical information about  the localized character of the partial derivatives of the continuous solution, we have no  meaningful pointwise bounds for the approximation error associated with any computational algorithm applied to  the Hemker problem. Here, we test for convergence of the numerical approximations using the double-mesh principle \cite{fhmos} and, more importantly, to identify when any numerical method fails to be convergent.
 We emphasize that we shall estimate the global pointwise convergence of the interpolated computed approximations across the entire domain. 
Numerical results with a preliminary version of this algorithm were reported in \cite{glasgow-BAIL}. 

The Shishkin mesh \cite{fhmos, mos} is a  central component of the algorithm. The simplicity of this mesh is one of the key attributes of this particular layer-adapted mesh, which allows easy extensions to more complicated problems. Shishkin meshes have the additional property that, if one has established  parameter-uniform nodal accuracy \cite{fhmos} in a subdomain, then (for  problems with regular exponential boundary layers \cite{hemkera} and characteristic boundary layers \cite{hemkerb}) this nodal accuracy extends to global accuracy across the subdomain using basic bi-linear interpolation.
This interpolation  feature of the mesh permits us employ overlapping subdomains, with different co-ordinate systems aligned to the local geometry of the layers, and to subsequently test computationally for global accuracy across the entire domain.

In \S 2, we identify bounds on the solution of the Hemker problem restricted to a bounded domain. 
In \S 3, we discuss how we computationally estimate the order of parameter-uniform convergence for any numerical method. 
 In \S 4, we construct and describe the numerical algorithm, which involves four distinct stages. The first two stages generate an initial approximation which has defects only near the characteristic points. The third and fourth stage correct this initial approximation.   In \S 5, we present numerical results to illustrate the performance of the final algorithm. The numerical results indicate that this new algorithm is generating numerical approximations which are converging, over an extensive range of the singular perturbation parameter,  to the solution of  a bounded-domain version of the original Hemker problem.   

{\bf Notation}: We will employ three distinct co-ordinate systems in this paper. A Cartesian co-ordinate system $(x,y)$, a polar co-ordinate system $(r,\theta)$ and a particular parabolic co-ordinate system  $(s,t)$. In each co-ordinate system, we adopt the following notation for functions:
 \[
 f(x,y)=\hat f(s,t) =\tilde f(r,\theta). \]
We use these co-ordinate systems to solve various sub-problems on an annulus $ A$, a rectangle $S$ and planar regions $ Q^+,  Q^-$. The numerical solution is determined using four sub-components $U_A, U_B, U_C$ and $U_D$; where $U_A$ is defined over  the annulus $A$, $U_C$  is defined over the planar region $ Q^+ \cup Q^-$ and $U_B, U_D$ are defined over the rectangular region $S$. In addition, the algorithm produces an initial global approximation $\bar U^N_1$, which is shown to lose accuracy in some of the layers. This initial approximation is subsequently corrected to produce a globally pointwise accurate approximation $\bar U^N_2$. Throughout the paper, $\Vert \cdot \Vert _D$ denotes the  supremum or maximum norm measured over the domain $D$.

\section{The continuous problem}

In this paper, we confine our discussion to the Hemker problem (\ref{Piet's-problem}) posed on a bounded domain of dimension $O(R^2)$.
Consider the singularly perturbed elliptic problem: Find $ u (x,y)$ such that 
\begin{subequations}\label{cont-prob}
\begin{eqnarray} 
 L u:=- \ve { \Delta} u +u_x =0, \quad (x,y) \in   \Omega ; \\
u=0, \  (x,y) \in \partial \Omega _O; \quad   u_x =0, \  (x,y) \in \partial \Omega _N;\quad
 u=1, \  (x,y) \in \partial \Omega _I; 
\end{eqnarray}
where the bounded domain $\Omega$ and the boundaries $\partial \Omega _O, \partial \Omega _N, \partial \Omega _I$ are defined to be
\begin{eqnarray}
\Omega := \{(x,y) \vert 1 < x^2 +y^2 <  R^2, x \leq 0\} \cup \{(x,y) \vert -R  < y <R, 0 < x < R, 1 < x^2 +y^2 \},
\\
\partial \Omega _N:= \{ (R,y) |-R < y < R \}, \quad
  \partial \Omega _I:= \{(x,y) \vert x^2 +y^2 =1 \}; \\
\partial \Omega _O:= \{ (x,y) \vert x^2 +y^2 =  R^2, x \leq 0\} \cup \{(x,y) \vert y =\pm R, 0 \leq x \leq R \}.
\end{eqnarray}
\end{subequations}
A sample computed solution (using the algorithm (\ref{corrected-approx})) is displayed in Figure \ref{Fig.0}, which illustrates the location of the layers that can appear in the solution. 
In all of our numerical experiments, we have simply  taken $R=4$. 
\begin{figure}  \center{\includegraphics[scale=0.15, angle=0]{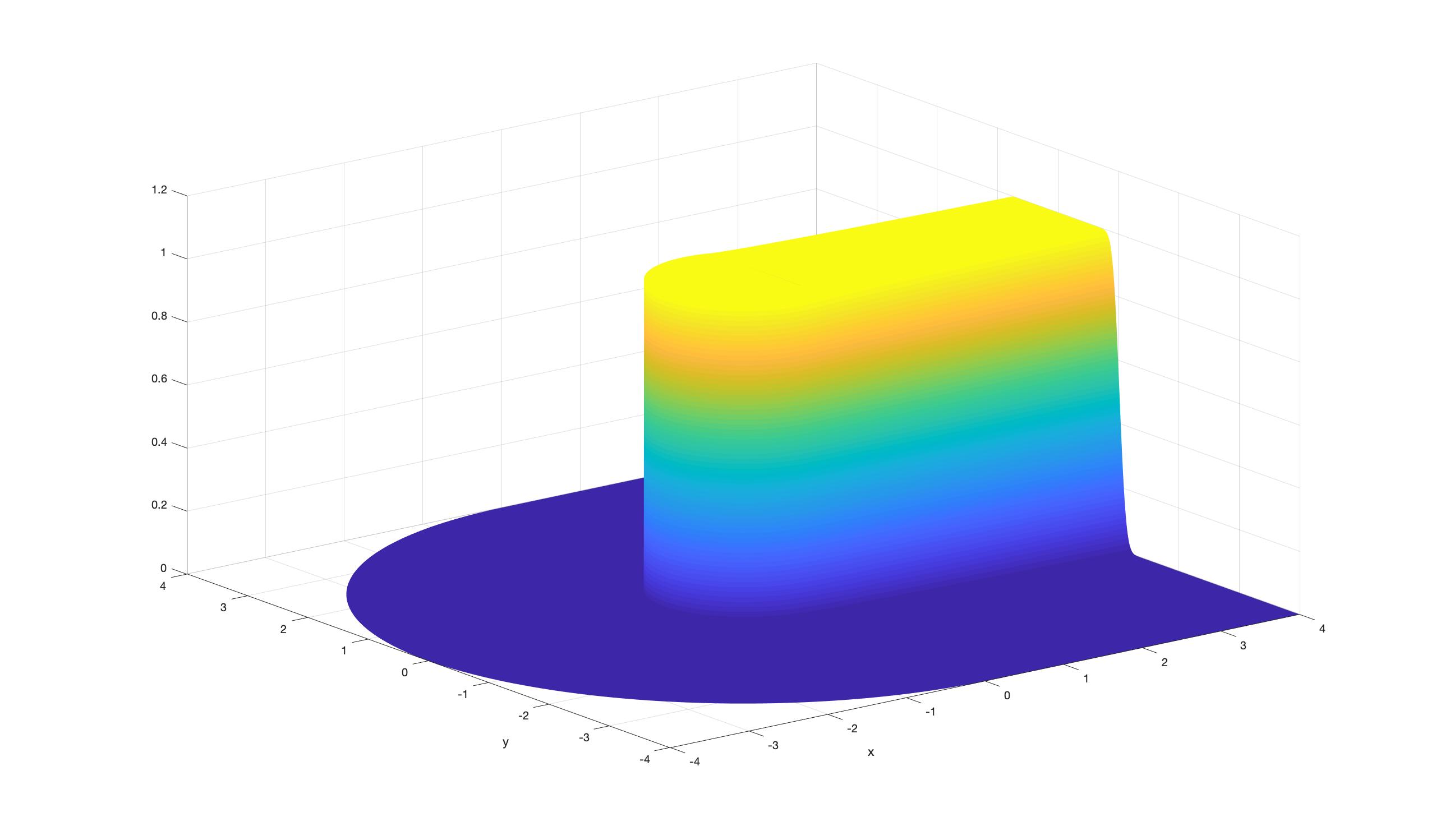}}
\caption{The computed solution $\bar U_2^{1024} $ of the Hemker problem (\ref{cont-prob}) generated by  numerical method (\ref{corrected-approx}) for $\ve =2^{-10}$}
\label{Fig.0}  
\end{figure}

We have a minimum principle associated with this problem.
\begin{theorem}\cite[pg. 61]{prot}\label{compar1}
For any $ w \in C^0(\bar D) \cap C^2(D), D \subset  \Omega$,  if $ L w (x,y) \geq 0, (x,y) \in D$, then
$\min _{\bar D}  w \geq \min _{\partial D} w$. 
\end{theorem}

\begin{proof}
 Consider the special case of $\min _{\partial D} w=0$. Assume $w<0$ at some internal point. Then the function
$ v := e^{-\frac{x}{2\ve} } w$ is negative at some interior point. However, then 
$ L  w = e^{\frac{x}{2\ve} }(-\ve { \Delta} v +\frac{1}{4\ve}  v) < 0$ at the point where $v$ takes its minimum value. This is a contradiction.
Complete the proof by considering the function $w_1:=w - \min _{\partial D} w$.
\end{proof}
Thus, we have a comparison principle. 
\begin{corollary} If $ w ,v \in C^0(\bar D) \cap C^2(D), D \subset  \Omega$ are such that  $ L  w (x,y)\geq  L v(x,y), \ \forall (x,y) \in  D$ and $ w \geq  v$ on the boundary $\partial D$, then $w (x,y) \geq v(x,y),\ \forall (x,y) \in \bar D$.
\end{corollary}
For any open connected subdomain $D \subset \Omega$ we define the boundaries
\[
\partial D _N:=  \{ \vec p \in \partial D \vert  (1,0)\cdot \vec n_p >0\} \quad
\hbox{and} \quad 
\partial D _O:= \partial D \setminus \partial D _N;
\]
where $\vec n_p$ is the outward normal to $\partial D$ at $\vec p$. 
As in \cite[pg.65]{prot}, we can establish

\begin{theorem}\label{compar2} If $ w , v \in C^0(\bar D ) \cap C^2(D ), D \subset  \Omega$ are such that  $ L w (x,y)\geq  Lv(x,y), \ \forall (x,y) \in  D$;  $ w \geq  v$ on the boundary $\partial D _O  $ and $ w_x \geq  v_x$ on the boundary $\partial D _N  $, then $ w (x,y) \geq  v(x,y),\ \forall (x,y) \in \bar D$.
\end{theorem}

 Let us partition the domain $\Omega $ into a finite number of non-overlapping subdomains $\{ D_i \} _{i=1}^n$ such that
\[
D_i \cap D_j = \emptyset , \  i \neq j, \quad \bar \Omega = \cup _{i=1}^n \bar D_i \quad \hbox{and} \quad \Gamma _{i,j} :=  \bar D_i \cap \bar D_j.
\]
 Let $\frac{\partial u }{\partial n_i}$ 
denote the outward normal derivative of each subdomain $D_i$ and
\[
 \ \bigl[ \frac{\partial u }{\partial n} \bigr] _{\Gamma _{i,j}}= \frac{\partial u }{\partial n_i}  \vert _{\Gamma _{i,j}} + \frac{\partial u }{\partial n_j} \vert _{\Gamma _{i,j}}.
\]
Using the usual proof by contradiction argument (with a separate argument for the interfaces $\Gamma _{i,j}$) we can establish the following 

\begin{theorem}\label{compar3} 
{ If $ w,v \in C^0(\bar \Omega ) \cap  C^2(\cup _{i=1}^nD_i )$  is such that 
(a)  $ L w (x,y)\geq   L v (x,y), \ \forall (x,y) \in  D_i$ for all $i$;  (b)   $ \bigl[ \frac{\partial w }{\partial n} \bigr] _{\Gamma _{i,j}} \geq \bigl[ \frac{\partial v }{\partial n} \bigr] _{\Gamma _{i,j}}$ for all $i,j$; 
(c) $ w \geq  v$ on the boundary $\partial \Omega _O  $ and  (d) $ w_x \geq  v_x$ on the boundary $\partial \Omega _N  $, then $ w (x,y) \geq  v(x,y),\ \forall (x,y) \in \bar \Omega$. }
\end{theorem}

\begin{proof}
 Consider the function
$ z := e^{-\frac{x}{2\ve} } (w-v)$ and assume that $z(\vec p) := \min _{\bar \Omega} z <0$. By (c), $z$ is not a constant function and $\vec p \not \in 
\partial \Omega _O$. Note that $(w-v)_x =e^{\frac{x}{2\ve} }(z_x+\frac{1}{2\ve}z)$ and so, by (d),  $\vec p \not \in \partial \Omega _N$. By (a),
$\vec p \not \in D_i$ for any $i$, as $L(w-v) =e^{\frac{x}{2\ve} }(-\ve \Delta z +\frac{1}{4\ve}z)$.
Finally,
\[
 \ \bigl[ \frac{\partial (w-v) }{\partial n} \bigr] _{\Gamma _{i,j}}= e^{\frac{x}{2\ve} } \ \bigl[ \frac{\partial z }{\partial n} \bigr] _{\Gamma _{i,j}}
\]
and, using the argument \cite[Theorems 7 and 8, pp. 65-67]{prot} over each subdomain $D_i$, we have that $\frac{\partial z }{\partial n}(\vec p)  <0$ if $\vec p \in \Gamma _{i,j}$. Hence we have the strict inequality $\bigl[ \frac{\partial z }{\partial n} \bigr] _{\Gamma _{i,j}}(\vec p) <0$, which contradicts (b).
\end{proof}

Using these results, we can establish the following bounds on the solution:
\begin{theorem}\label{main} Assuming $\ve$ is sufficiently small, then the solution $u$ of problem (\ref{cont-prob}) satisfies the following bounds
\begin{subequations}\label{solutions-bounds}
\begin{eqnarray} 
0 \leq  u(x,y) \leq 1, && (x,y) \in \bar \Omega ; \label{basic}\\
\tilde u(r,\theta) \leq Ce^{\frac{\cos (\theta ) (r-1)}{\ve}}, && \cos \theta \leq   0,\ r \geq 1;
\label{left}\\
 u (x,y)\leq Ce^{-\frac{ (\vert y \vert -1)}{\sqrt{\ve}}}, && \vert y \vert  \geq 1,\ -R < x< R; \label{right}\\
 u(x,y) \leq Ce^{-\frac{(0.5x^2+\vert y \vert -1)}{3\ve ^{2/3}}}, &&   x \in \ve ^{1/3}[-1,1],\    -C \ve ^{2/3} \leq 0.5x^2+\vert y \vert -1  \leq C \ve ^{2/3}. \label{center} 
\end{eqnarray}
\end{subequations}
\end{theorem}

\begin{proof}
The first bound follows easily from the minimum principle. See details in the appendix for all of the remaining bounds. 
\end{proof}

\section{Computationally testing for  convergence}

As identified in \cite{hemker}, we do not have a useable closed form representation of the exact solution to the Hemker problem to allow us  to evaluate the accuracy of any computed approximation. The infinite series representation \cite{hemker}
for the exact solution  has difficulties for moderately small values of the singular perturbation parameter $\ve$. Hence, to test for  convergence we rely on the double-mesh method of estimating the order of convergence \cite[Chapter 8]{fhmos}. We elaborate on this experimental approach in this section.

For every particular value of $\ve$ and $N$, let  $U_\ve^N$ be the computed solutions  on certain  meshes $\Omega _\ve ^N$, where $N$ denotes the number of mesh elements used in each co-ordinate direction.  Define the maximum local two-mesh global differences $ D^N_\ve$ and the maximum parameter-uniform two-mesh global differences $D^N$ by
\footnote{{In passing we note that, in general, for a piecewise-uniform Shishkin mesh $ \Omega _\ve^N \cup \Omega _\ve^{2N} \neq \Omega _\ve^{2N}$, as the transition point (where the mesh is not uniform) depends on $N$.}}
\[
D^N_\ve:= \Vert \bar U_\ve^N-\bar U_\ve ^{2N}\Vert _{\Omega _\ve^N \cup \Omega _\ve^{2N}} \quad \hbox{and}  \quad D^N:= \sup _{0 < \ve \leq 1} D^N_\ve,
\]
where $\bar U_\ve ^N$  denotes the bilinear interpolation of the discrete solution $U_\ve ^N$ on the mesh $\Omega _\ve ^{N}$. 
 Then, for any particular value of $\ve $ and $N$, the local orders of global convergence are denoted by $\bar p ^{N}_\ve$ and, for any particular value  of $N$ and {\it all values of $\ve $}, the  {\bf parameter-uniform} global orders of   convergence $\bar p ^N$ are defined, respectively,  by 
\[
 \bar p^N_\ve:=  \log_2\left (\frac{D^N_\ve}{D^{2N}_\ve} \right) \quad \hbox{and} \quad  \bar p^N :=  \log_2\left (\frac{D^N}{D^{2N}} \right).
\]
If, for a certain class ${\cal{C}}$ of singularly perturbed problems, there exists a theoretical error bound of the form: There exists a constant $C$ independent of $\ve$ and $N$ such that for all $\ve >0$
\begin{equation}\label{theory}
\Vert \bar U^N -u \Vert _{\Omega}\leq CN^{-p},\quad p >0;
\end{equation}  then it follows that 
\[D^N \leq C(1+2^{-p}) N^{-p}. \]
Hence, for any particular sample problem from this class ${\cal{C}}$, we expect to observe this theoretical convergence rate $p$ in the computed rates of convergence $\bar p^N $. That is, we expect that $\bar p^N \approx p$. 

A useful attribute of the two-mesh method is that we can use it to identify when a numerical method is not parameter-uniform.
Observe that 
\[
\Vert \bar U_\ve ^N - \bar U_\ve ^{2N} \Vert _{\Omega} \leq \Vert \bar U_\ve^N - u \Vert_{\Omega} +\Vert u - \bar U_\ve ^{2N} \Vert_{\Omega}.
\]
Hence, if the parameter-uniform two mesh differences $D^N$ fail to converge to zero, then the numerical method is also not a parameter-uniform numerical method. 
In our quest for a parameter-uniform numerical method, we used this feature to identify necessary components to construct  a parameter-uniform numerical method. 
On the other hand, without the existence of a theoretical error bound (\ref{theory}) (as is the case with the Hemker problem),  if the global two mesh differences $D^N$ are seen to converge then we can only conclude that the numerical method may be a parameter-uniform numerical method. 
We would require a theoretical parameter-uniform  error bound on the numerical approximations, before one can  assert that the numerical method is indeed parameter-uniform. 

For any numerical method applied to  a class of singularly perturbed problems, our primary interest is in determining the parameter-uniform orders of global convergence $\bar p^N$. However, we can also examine the local orders of convergence $\bar p_\ve^N$ to see how the numerical method performs for each possible value of $\ve$ over the range $0<\ve \leq 1$. In general, we note that $\bar p^N \neq \min _{\ve}  \bar p^N_\ve$. 
In the case of piecewise-uniform meshes, certain anomalies can sometimes be observed in the local orders of convergence (i.e., $\bar p_\ve ^N \not\approx \bar p^N$) for certain values of $\ve$. 
We illustrate this effect with  the following theoretical example. Based on the nature of the typical errors on a piecewise-uniform Shishkin mesh in a one dimensional convection-diffusion problem, suppose that the two-mesh differences $D_\ve^N$ were of the form
\[
D_\ve ^N := \Bigl\{ \begin{array}{ll} \frac{\rho}{1+\rho} \qquad   \hbox{if} \quad k\ve \ln N \geq 1 \\
N^{-1} \qquad \hbox{if} \quad k\ve \ln N < 1
\end{array} ; \qquad  \hbox{where} \quad \rho := \frac{1}{\ve N}, \quad k \geq 1 .
\]For this theoretical example, the parameter-uniform two mesh differences can be explicitly determined. 
Note first that 
\[
D_\ve ^N  \leq \Bigl\{ \begin{array}{ll} \frac{k\ln N}{N+k\ln N} \qquad \qquad  \hbox{if} \quad k\ve \ln N \geq 1 \\
N^{-1} \qquad \qquad  \hbox{if} \quad k\ve \ln N < 1
\end{array} .
\]
Hence, if $N\geq 4, k \geq 1$, we have 
\[
D^N =\frac{k\ln N}{N+k\ln N} \quad \hbox{and} \quad \lim _{N\rightarrow \infty} \bar p^N =1. 
\]
Let us now consider the particular values of $\ve = 2^{-4},k=4$ and $N=32$. First observe that
\[
D_{\ve =2^{-4}} ^N  = \Bigl\{ \begin{array}{ll} \frac{\rho _*}{1+\rho _*} \qquad \qquad  \hbox{if} \quad  \ln N \geq 4 \\
N^{-1}\qquad \qquad  \hbox{if} \quad   \ln N < 4
\end{array} , \qquad \hbox{where} \quad \rho _* = 16N^{-1}.
\]
Then, in particular,
\[
D_{\ve =2^{-4}} ^{32} = 2^{-5}, D_{\ve =2^{-4}} ^{64} =0.2,
\]
which yields the order of convergence as 
\[
\bar p_{\ve =2^{-4}} ^{32} \approx -2.68 \quad \hbox{although}\quad  \bar p ^{32} \approx 0.55.
\]
Thus, we can have negative local orders of  convergence  $\bar p^N_\ve$, for particular values of $\ve $ and $N$ and still have positive parameter-uniform orders of convergence
$\bar p^N$. This phenomena will appear in the numerical results section in \S 5.  

In practice, note that the parameter-uniform orders $\bar p^N$ can only be estimated over a finite set $R^J_\ve := \{2^{-j}, j=0,1,\ldots J\}$ of values of the singular perturbation parameter 
$\ve \in (0,1]$. That is, we define
\[
\bar p^N_{R^J_\ve} :=  \log_2\left (\frac{D_{R^J_\ve}^N}{D_{R^J_\ve}^{2N}} \right), \quad \hbox{where} \quad  D_{R^J_\ve}^N:= \max _{\ve \in R^J_\ve } D^N_\ve.
\]
When a method is known to be parameter-uniform, $J$ is taken sufficiently large so $\bar p^N_{R^K_\ve} = \bar p^N_{R^J_\ve}$, for any $K >J$ and $\bar p^N_{R^J_\ve}$ is taken to be the computed estimate of 
$\bar p^N$.

In this paper, we construct a numerical method that  displays positive  orders of convergence $\bar p^N_{R^{20}_\ve}$, for $N$ sufficiently large; i.e.,  $N \geq N_0$, where $N_0$ is independent of $\ve$, when the numerical method is applied to the Hemker problem and the range of the singular perturbation parameter is $\ve \in R_\ve:= \{2^{-j}, j=0,1,\ldots 20\}$.
For smaller values of the parameter ($\ve < 2^{-24}$), we have observed a degradation in the local orders of  convergence. Hence, we cannot claim that the numerical method described in this paper  is parameter-uniform.  This lack of convergence might be due to the presence of an unidentified singularity; but, in effect,  the character of the solution for $\ve < 2^{-20}$ remains an open question.

To conclude this section, we note that the two-mesh differences are computed to enable the computation of approximate orders of convergence. To generate approximations to the  pointwise errors for any particular value of $N$ and $\ve$,  we compare the computed solution for a given number of mesh points $N$ to the computed solution on a fine mesh.
That is, we approximate the nodal error by 
\[
\Vert U_\ve^N-u \Vert _{\Omega ^N_\ve } \approx \Vert  U_\ve^N-\bar U_\ve^{4N_*} \Vert _{\Omega _\ve^N }, \quad \hbox{for any} \quad N \leq N_*; \]
and the global error by
\[
\Vert \bar U_\ve^N-u \Vert _{\Omega } \approx \Vert \bar U_\ve^N-\bar U_\ve^{4N_*} \Vert _{\Omega _\ve^N \cup  \Omega _\ve^{4N_*}}, \quad \hbox{for any} \quad N \leq N_*.
\]
\section{The numerical algorithm}

Polar coordinates are a natural co-ordinate system to employ  in the semi-annular region to the left of the line $x=0$ and 
rectangular co-ordinates are natural to the right of the line $x=0$. To incorporate both co-ordinate systems, we first generate an approximate solution to the problem  (\ref{cont-prob}) on the sector
\begin{equation}\label{annulus}
\tilde A:= \{(r,\theta) \vert 1< r < R, \frac{\pi}{2} - \tau\leq  \theta  \leq\frac{3\pi}{2} +\tau\},   
\end{equation}
 which is a proper subset of the domain $\Omega$ and the parameter $\tau$ is specified below in (\ref{tau-1}).

The continuous problem (\ref{cont-prob}), restricted to the sector $\tilde A$, is transformed into the problem: Find a periodic function, $\tilde u(r,\theta) =  u(x,y)$ such that
\begin{subequations}\label{polar}
\begin{eqnarray}
\tilde  L\tilde u:= -\frac{\ve }{r^2}\tilde  u _{\theta \theta}  -\ve  \tilde u_{rr} +  \bigl(\cos (\theta ) -{\frac{\ve}{r}} ) \tilde u _r  - \frac{\sin (\theta )}{r}  \tilde  u _{\theta} = 0,  \ \hbox{in } \tilde A ;\\
\tilde u(1,\theta) =1, \quad \frac{\pi}{2} -\tau \leq \theta  \leq \frac{3\pi}{2} +\tau ; \qquad  \tilde  u(R,\theta) =0, \quad { \frac{\pi}{2} \leq \theta  \leq \frac{3\pi}{2} }.
\end{eqnarray}
The  remaining boundary points of this sector  $\tilde A$ 
are internal points within the domain $\Omega$, where the discrete solution has not yet been specified. 
In order to generate an initial approximation to the solution $\tilde  u$, we  impose 
 homogeneous Neumann conditions at these internal points of the form:
\begin{eqnarray}
 u^A_x (x,y)=0, \quad \hbox{if} \quad x^2+y^2=R^2 \quad \hbox{and} \quad x >0; \\
 \tilde u^A_\theta (r,\frac{\pi}{2}-\tau  )= \tilde u^A_\theta (r,\frac{3\pi}{2}+\tau  )=0, \quad \hbox{for} \quad 1 < r < R,
\end{eqnarray}
\end{subequations}
where $u^A(x,y) \approx u(x,y)$ for $x \leq 0$.

\begin{remark}The choice of a  homogeneous Neumann condition at the outflow is motivated by the following observation:
Consider the one-dimensional convection-diffusion problem: Find $z(x), x \in [0,L]$ such that
\[
-\ve z'' +az'=f(x), \ x \in (0,L);\quad  z(0)=A, z(L)=B; \quad a(x) \geq \alpha >0
\]
and the approximate problem: Find $z_A(x), x \in [0,L]$ such that
\[
-\ve z_A'' +az_A'=f(x), \ x \in (0,L);\quad  z_A(0)=z(0), z'_A(L)=0.
\]
Using a comparison principle (as in Theorem \ref{compar2}), we can establish the bound
\[
\vert (z-z_A)(x) \vert \leq \frac{\ve \vert z'(L)\vert}{\alpha } e^{-\frac{\alpha(L-x)}{\ve}}.
\]
Hence, $z_A$ is an $O(N^{-1})$-approximation to $z$, at some $O(\ve \ln N)$ distance away from the end-point $x=L$. That is:
\[
\vert (z-z_A)(x) \vert \leq C\ve \vert z'(L)\vert N^{-1}\leq CN^{-1}; \quad \hbox{if} \quad x \in [0, L-\frac{\ve \ln N}{\alpha}].
\]
\end{remark}

This problem (\ref{polar}) is discretized using simple upwinding on a tensor product piecewise-uniform Shishkin mesh,  whose construction is motivated by the bounds (\ref{left}), (\ref{center}).  Two transition points (where the mesh step changes in magnitude) are used in the radial direction. The  choice of the first transition point $\sigma _1$ is motivated by considering the bound (\ref{left}) at some fixed distance $x < - \delta <0$ to the left of $x=0$ and by the theoretical error bounds \cite{annulus} established  for this mesh in the region where $x < - \delta <0$. The choice of this point $x=-\delta$ is arbitrary.  As in \cite{annulus},  we simply take it to correspond to a criticial angle $\theta _*$, such that
\[
\kappa \cos \theta _* = -\frac{1}{2}, \quad 0.5 < \kappa \leq 1,
\]
where $\kappa$ is again arbitrary. 
Hence, we have that
\[
\tilde u (r,\theta) \leq C e^{-\frac{r-1}{2\ve}}, \quad \hbox{for} \quad \theta _* \leq \theta \leq 2\pi - \theta _*
\]
and
\[
\tilde u (r,\theta) \leq CN^{-1}, \quad \hbox{if} \quad r \geq 1+2\ve \ln N \quad \hbox{and} \quad \theta _* \leq \theta \leq 2\pi - \theta _*.
\]
A second transition point $\sigma _2$ is motivated by the bound (\ref{center}) applied along the line $x=0$.  In the angular direction, the bound (\ref{center}) also motivates the inclusion of a transition point $\tau$  in the vicinity of the characteristic points. See also \cite[pg. 269]{eckhaus}, \cite[pg. 188]{ilin} and \cite[pg. 1183]{waechter} for motivation for these scales in the vicinity of the characteristic points $(0,\pm 1)$. 

{\bf The Shishkin mesh $\tilde  \Omega ^{N}_A$}
{\it 
The radial domain $[1,R]$ is divided into three subregions. The  radii 
$
r=1, \ r=1+\sigma _1, \ r=1+ \sigma _1 +\sigma _2, \  r=R,
$
mark the subregion boundaries and  the radial transition points are taken to be
\begin{subequations}\label{fitted-mesh-special}
\begin{equation}
  \sigma _1:= \min \{ \frac{R-1}{4},  2\ve  \ln N   \} \quad \hbox{and} \quad
 \sigma _2:= \min \{ \frac{R-1}{4}, 3 \ve ^{2/3} \ln N  \}.
\end{equation}
The $N$ radial mesh points are distributed in the ratio $N/4:N/4:N/2$ across these three subintervals. 
For the angular coordinate, 
the interval $[\frac{\pi}{2}-\tau,\frac{3\pi}{2}+\tau ]$ is split into three subintervals with the start/end points of each subinterval, respectively, at
\[
\frac{\pi}{2} -\tau, \frac{\pi}{2} +\tau, \frac{3\pi}{2} -\tau, \frac{3\pi}{2} +\tau 
\] 
 and the mesh points are distributed in the ratio $N/4:N/2:N/4$ 
across  the three associated subintervals.  The transition points are determined by 
\begin{equation}\label{tau-1}
 \tau := \min \{ \frac{\pi}{6}, \sqrt{6} \ve^{1/3} \ln N   \}.
\end{equation}
\end{subequations}
}
A schematic  image of this mesh is presented in Figure \ref{Fig.1}. 
  However,  in practice,  the refinement in the radial direction only becomes apparent to the user for very small values of $\varepsilon$.
\begin{figure}  \center{
\includegraphics[width=0.5\textwidth,height=0.5\textwidth]{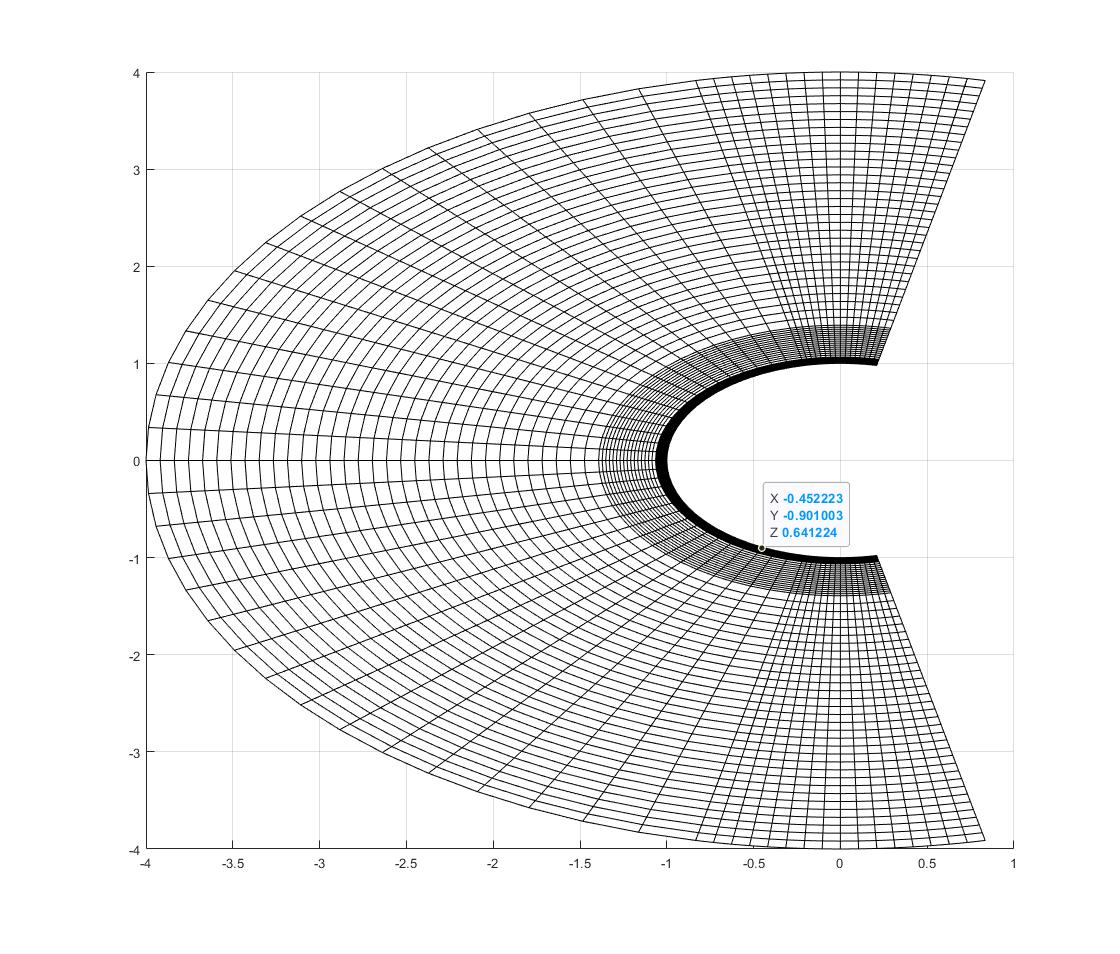}}
\caption{A schematic image of the  mesh $\tilde \Omega _A^N$ (\ref{fitted-mesh-special}) on the  annular subregion
$\tilde A$ (\ref{annulus}) }
\label{Fig.1}  
\end{figure}

 At the mesh points on the sector $ \tilde A$, the computed solution will (in polar coordinates) be denoted by $U_A(r_i,\theta_j), \ (r_i,\theta_j) \in \bar \Omega _A^{N}$. This approximation is extended to the global approximation $\bar U_A(r,\theta),\ (r,\theta) \in \bar A$, using simple bilinear interpolation
{ \footnote{
Over any computational cell $\Omega _{i,j}:= (x_{i-1},x_i)\times (y_{j-1},y_j)$, where  $h_i:=x_i-x_{i-1}, k_j:=y_j-y_{j-1}$,  we denote the bilinear interpolant of any function $g(x,y)$ by $\bar g$.
For any smooth function  $v$, where  { $\Vert v_{xx} \Vert _{\Omega _{i,j}}+ \Vert v_{yy} \Vert _{\Omega _{i,j}}\leq C$, then
$
\Vert v -\bar v \Vert _{\Omega _{i,j}}\leq C h_i^2 +C k_j^2.
$}
For a layer function of the form $w(x,y) = e^{-\alpha x/\ve ^p}$, 
\begin{eqnarray*}
{ 
\Vert w -\bar w \Vert _{\Omega _{i,j}}
\leq C \Vert w \Vert _{\Omega _{i,j}}} \leq CN^{-1}, \quad \hbox{if} \quad x_i \geq \frac{\ve ^p \ln N}{\alpha}, \quad \hbox{and}  \\
{ \Vert w -\bar w \Vert _{\Omega _{i,j}}\leq C h_i\Vert w _x\Vert _{\Omega _{i,j}} }
\leq CN^{-1}\ln N, \quad \hbox{if} \quad h_i \leq C\ve ^p N^{-1} \ln N.
\end{eqnarray*}
}}  

We will utilize the following 
finite difference operators  $D^+_r,D^-_r, D^\pm_r, \delta^2_r$  defined by
\begin{eqnarray*}
D^+_rZ(r_i,\theta _j) :=\frac{Z (r_{i+1},\theta _j)-Z(r_i,\theta _j)}{r_{i+1}-r_i},\quad  D_r^-Z(r_i,\theta _j) :=\frac{Z(r_{i},\theta _j)-Z (r_{i-1},\theta _j)}{r_{i}-r_{i-1}}; \\
2(bD_r^\pm)Z := (b-\vert b \vert) D_r^+Z + (b+\vert b \vert) D_r^-Z; \quad \delta^2_r Z(r_i,\theta _j) :=\frac{D^+_rZ(r_{i},\theta _j)-D_r^-Z(r_i,\theta _j)}{(r_{i+1}-r_{i-1})/2}.\end{eqnarray*}

{\bf Stage 1}: {\it Numerical method for $U_A$ defined over an annular subregion $\tilde A$}

Find $U_A$ such that:
\begin{subequations}\label{discrete-problem-A}
\begin{eqnarray} 
 \tilde L^{N} U_A=  0,\qquad (r_i,\theta_j) \in  \Omega ^{N}_A; \\ \hbox{where} \quad 
 \tilde L^{N} U:= -\frac{\ve}{r_i^2}  \delta ^2 _{\theta}U   -\ve  \delta ^2 _{r}U  +  (\cos (\theta _j )-\frac{\ve}{r_i} ) D^{\pm}_r U  -\frac{\sin (\theta _j )}{r_i}D^{\pm}_\theta U; 
\\
U_A(1,\theta _j) =1,\ { \frac{\pi}{2} -\tau \leq \theta _j \leq \frac{3\pi}{2} +\tau ;} \quad U_A(R,\theta _j) =0,  \ \frac{\pi}{2} \leq \theta _j \leq \frac{3\pi}{2}; \\
(\cos \theta  D^-_r-\frac{\sin \theta}{R}  D^\pm_\theta \bigr) U_A(R,\theta _j) =0,  \quad   \theta _j \in (\frac{\pi}{2} -\tau , \frac{\pi}{2}) \cup  (\frac{3\pi}{2}, \frac{3\pi}{2} +\tau );\\
D^+_\theta U^A (r_i,\frac{\pi}{2}-\tau  )= D^-_\theta U^A (r_i,\frac{3\pi}{2}+\tau )=0, \quad \hbox{for} \quad 1 < r_i < R.
\end{eqnarray}
\end{subequations}
The boundary condition (\ref{discrete-problem-A}d) corresponds to applying the Neumann condition $u_x=0$ at these internal points of $\bar \Omega$. 

In Table \ref{Tab1}, we present the results from applying this numerical method to  problem (\ref{cont-prob}) posed on the sector $ \bar A \subset \Omega$. We display the orders of convergence only  for the region where $x \leq 0$ and we observe global convergence over the parameter range $\ve \in [2^{-20},1]$.  
\begin{table}[ht]
\centering\small
\begin{tabular}{|c| c c c c c c c|}
\hline
\multicolumn{8}{|c|}{$\bar p^{N}_{\epsilon}$}\\[3pt]
\hline$\ve | N$&N=\bf{8}&\bf{16}&\bf{32}&\bf{64}&\bf{128}&\bf{256}&\bf{512}\\[3pt]
\hline
  $\ve =1$     &  0.9855 &    0.9966&    0.9996 &     0.9996&   1.0002 &    1.0001 &   1.0000\\ 
  $2^{-2}$ &  1.0014 &    0.9115 &    0.9643&    0.9867 &    0.9918 &    0.9958 &    0.9979 \\
  $2^{-4}$ & 0.4693 &   0.6789 &    0.7825 &    0.6875 &    0.6950&    0.7541 &    0.9911 \\
 $2^{-6}$  & 0.5832 &   0.7176&    0.7760 &    0.6887 &    0.7691 &    0.7990 &    0.8255 \\
 $2^{-8}$  &  0.7205&   0.7771 &    0.8441 &    0.6893 &    0.7713&    0.7991&    0.8266 \\
 $2^{-10}$ &  0.3249 &    0.7079&    0.9677&    0.8376 &    0.9761 &    0.9231 &    0.8270 \\
 $2^{-12}$ &  0.1131&    0.4659&    0.8732 &    0.9145&    0.8086&    0.9597&    1.0268 \\
 $2^{-14}$ &  0.1242 &    0.5344 &    0.6122&    0.7242 &    0.8605&    0.7995&    0.9612\\
 $2^{-16}$ &  0.1669 &    0.4506 &    0.4704 &    0.6033 &    0.7562&    0.8573 &    0.8242 \\
 $2^{-18}$ &  0.2113&    0.2972 &    0.3714&    0.5046 &    0.6604&    0.7978&    0.8561\\
 $2^{-20}$ &  0.2111&    0.1873&    0.2770 &    0.4167&    0.5625&    0.7077 &    0.8373\\
\hline
$\bar p^N_{R^{20}_\ve}$&  0.2111&    0.1873&    0.2770 &    0.4167&    0.5625&    0.7077 &    0.8373\\
\hline
\end{tabular}
\caption{Computed double-mesh global orders of convergence $\bar p^N_{\ve}$ using the mesh $\tilde \Omega ^{N}_A$ (\ref{fitted-mesh-special})  with $x\le 0$, when applied to  problem (\ref{cont-prob})  confined to the sector $\tilde A$  with $R=4$.}
\label{Tab1}
\normalsize
\end{table}

We next introduce a rectangular mesh, which will be aligned to the internal characteristic layers. We retain the computed solution $\bar U_A$ in the upwind region where  $x \leq 0$ and then solve the  problem (\ref{cont-prob}) over the remaining rectangle
\begin{equation}\label{rectangle}
S:= \{ (x,y) | 0 < x \leq R ,-R \leq y \leq R\},
\end{equation}
 using a piecewise-uniform mesh 
$ \Omega ^N_S$, whose transition parameters are related to the  bounds  (\ref{right}) on the continuous solution $u$.  { By  (\ref{right})}, 
\[
u(x,y) \leq CN^{-1}, \quad \hbox{if} \quad y \geq 1 +2{ \ve ^{1/2}}\ln N.
\]

{\bf The Shishkin mesh $ \Omega ^{N}_S$}
{\it The mesh 
$ \Omega ^N_S:= \omega _u \times \omega _3$ is a tensor product mesh of a uniform mesh $\omega _u $ in the horizontal direction and a Shishkin  mesh $\omega _3$,  which refines in the region of the interior characteristic layers. The mesh $\omega _3$
is generated by splitting the vertical interval $[-R,R]$ into the five subregions
\begin{subequations}\label{Shish-par-mesh}
\begin{equation}  [-R,-1 -\tau_2 ] \cup [-1 -\tau _2, -1 +\tau _1] \cup [-1 +\tau _1, 1-\tau _1] \cup [1-\tau_1, 1+\tau _2] \cup [1+\tau _2, R ],
\end{equation}
 distributing the mesh elements in the ratio $N/8:N/4:N/4:N/4:N/8$ and  
\begin{equation} 
\tau _1  := \min \{ \frac{1}{2},   2\ve ^{1/2} \ln N  \}; \quad 
\tau _2 :=  \min \{ \frac{R-1}{2}, 2 \ve ^{1/2}  \ln N  \} .
\end{equation}
\end{subequations}
}  

 Observe that some of the mesh points in $ \Omega ^{N}_S$ lie within the unit circle, where the value of the continuous solution is known. 

{\bf Stage 2}: {\it Numerical method for $U_B$ defined over  the downwind region $S$:}

  Find  $U_B(x_i,y_j)$  such that
\begin{subequations}\label{discrete-problem-B}
\begin{eqnarray} 
L^{N,M} U_B:= \Bigl( -\ve  \delta ^2 _{x}   -\ve  \delta ^2 _{y}  + D^{-}_x \Bigr)  U_B(x_i,y_j)     = 0 ,  \quad (x_i,y_j) \in \Omega ^N_S \setminus {  C_1};\\
U_B(x_i,y_j)\equiv 1, \quad (x_i,y_j) \in \Omega ^N_S \cap  {  C_1}; \quad C_1:=\{ (x,y) | x^2+y^2 \leq 1\}; 
 \end{eqnarray} with the remaining boundary values computed from the equations 
\begin{eqnarray} 
D^-_x  U_B(R,y_j)=0,  \ -R < y_j < R,\quad U_B(x_i,-R)= U_B(x_i,R)=0, \ x_i \in [0,R]; \\ U_B(0,y_j) = \bar U_A (0,y_j),\quad  y_j \in (-R, R) \setminus [-1,1].
\end{eqnarray}\end{subequations}
Here $\bar U_A (0,y_j)$ is a linear interpolant of the values $U_A (r_i,\theta _j)$ along the line $x=0$.

The initial computed  global approximation $\bar U_1^N$ to the solution  of problem (\ref{cont-prob})  is:
\begin{equation}\label{initial-approx}
\bar U_1^N (x,y) := \left \{
\begin{array}{ll}
\bar U_A(r,\theta ), & \hbox{for} \quad (r,\theta ) \in \overline {\tilde A} \setminus \{ x\geq 0 \} \\
\bar U_B (x,y), & \hbox{for} \quad (x,y) \in \bar S \setminus {  C_1}, 
\end{array}\right.
\end{equation}
where $U_A$ is defined by (\ref{discrete-problem-A}) and $U_B$ is defined by (\ref{discrete-problem-B}). 
\begin{figure}\center{
\includegraphics[width=0.5\textwidth,height=0.5\textwidth]{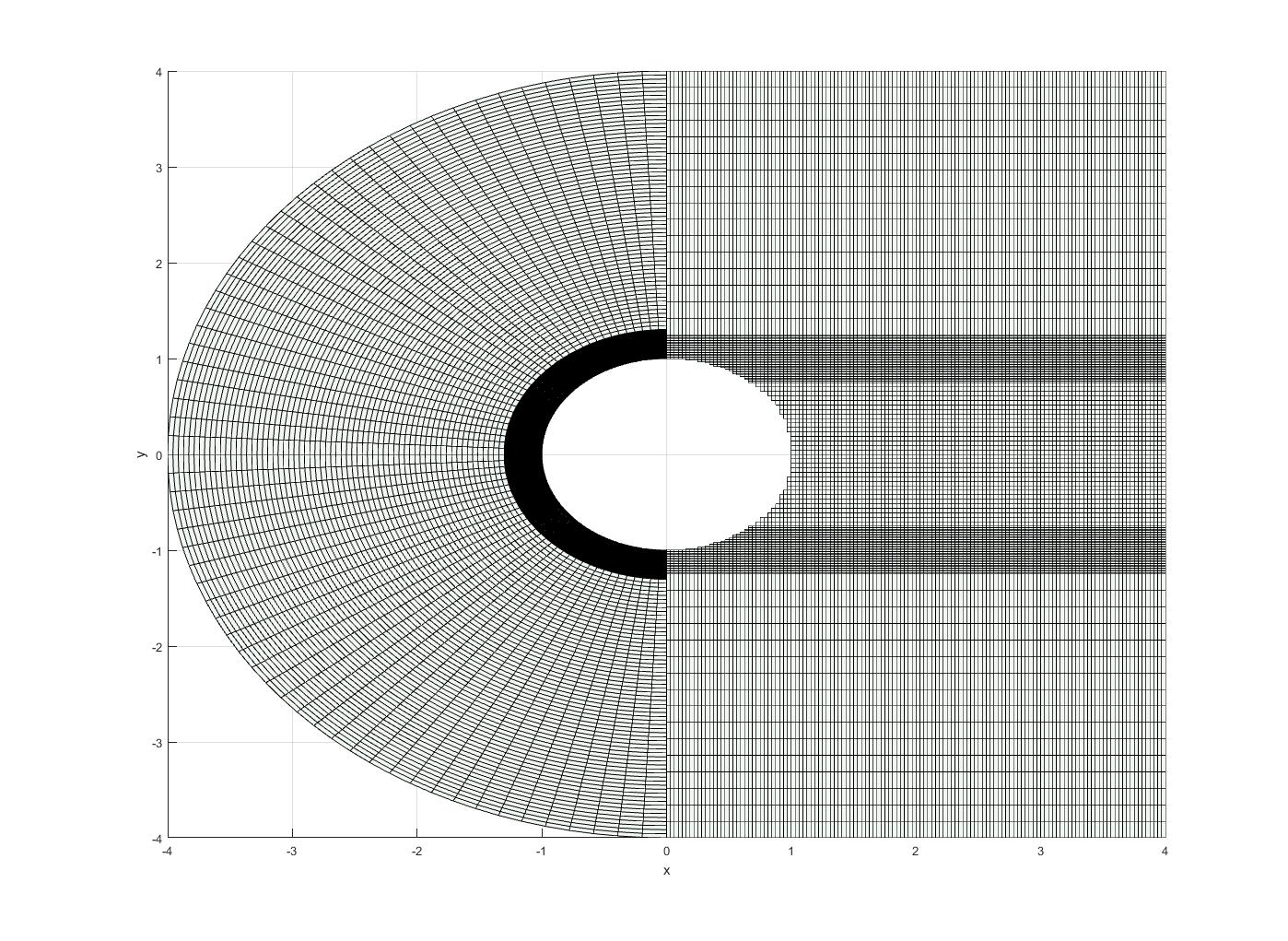}}
\caption{A schematic image of the composite  mesh 
{ $(\tilde \Omega _A^{N} \cup  \Omega ^N_S) \setminus {  C_1}$, defined by }(\ref{fitted-mesh-special}) and (\ref{Shish-par-mesh}) }\label{Fig.2} 
\end{figure}

In Table \ref{Tab2},  we do not observe  convergence of these initial numerical approximations $\bar U_1^N$.  Hence, although we observe convergence in the annulus up to $x \leq 0$, this does not suffice to generate convergence across the entire domain, even if the mesh is fitted to the characteristic layers. 
In Figures \ref{Spike-N128}, \ref{Spike-N256}   we plot the  error across the entire domain and we observe 
a spike in the global pointwise error in the vicinity of the characteristic points. This global error does not decrease when the mesh is refined. 

 We next describe the construction of a  correction $\bar U_2^N$ to the initial approximation $\bar U_1^N$. 

 \begin{table}[ht]
\centering\small
\begin{tabular}{|c| c c c c c c c|}
\hline
\multicolumn{8}{|c|}{$\bar p^{N}_{\epsilon}$}\\[3pt]
\hline$N$&\bf{8}&\bf{16}&\bf{32}&\bf{64}&\bf{128}&\bf{256}&\bf{512}\\[3pt]
\hline
 $\ve =1$     &   1.1831 &   0.5014&    0.7157   & 0.7719&    0.8259&    0.8743&    0.9100 \\
  $2^{-2}$ &    1.7592  &  0.0305 &   0.4492 &   0.6880  &  0.7791  &  0.8460    &0.8944\\
  $2^{-4}$ &    1.0764 &   1.0413   & 0.0241  &  0.4368  &  0.6719  &  0.7768 &   0.8486\\
  $2^{-6}$ &    0.2376  &  0.7503   & 1.7967  & -0.9041  &  0.7151&   0.4173  &  0.7648\\
  $2^{-8}$ &   -0.0726  &  0.3699   & 0.6842  &  0.1372 &   0.8439  & -0.0165  &  0.6041\\
  $2^{-10}$ &   -0.0700  &  0.0504 &   0.4604  &  1.0622 &   0.4357 &   0.2011 &   0.4706\\
  $2^{-12}$ &  -0.1139&   -0.0051  & 0.2692 &   0.8506   & 1.0414 &  -0.0081  &  0.3631\\
   $2^{-14}$  &-0.1508 &  -0.0405   & 0.1558   & 0.5954  &  1.2268   & 0.1912   & 0.2453\\
  $2^{-16}$&  -0.1680   &-0.0685   & 0.0477 &   0.3999  &  1.0007   & 0.7912  &  0.0780\\
  $2^{-18}$ &   -0.1686   &-0.0873&   -0.0159  &  0.2119  &  0.7662   & 1.3778  & -0.0380\\
  $2^{-20}$ &     -0.1690  & -0.0922 &  -0.0467  &  0.0795   & 0.5156   & 1.2055   & 0.5538\\
\hline
\end{tabular}
\caption{Computed double-mesh global orders of uniform convergence $\bar p_\ve^{N}$,  for the  mesh $\tilde \Omega ^N_A$ (\ref{fitted-mesh-special}) used up to $x=0$ and subsequently combined with the rectangular  mesh $\Omega ^N_S$ (\ref{Shish-par-mesh}) }
\label{Tab2}
\normalsize
\end{table}

 \begin{figure}[h!] \centering
		\includegraphics[scale=0.18, angle=0]{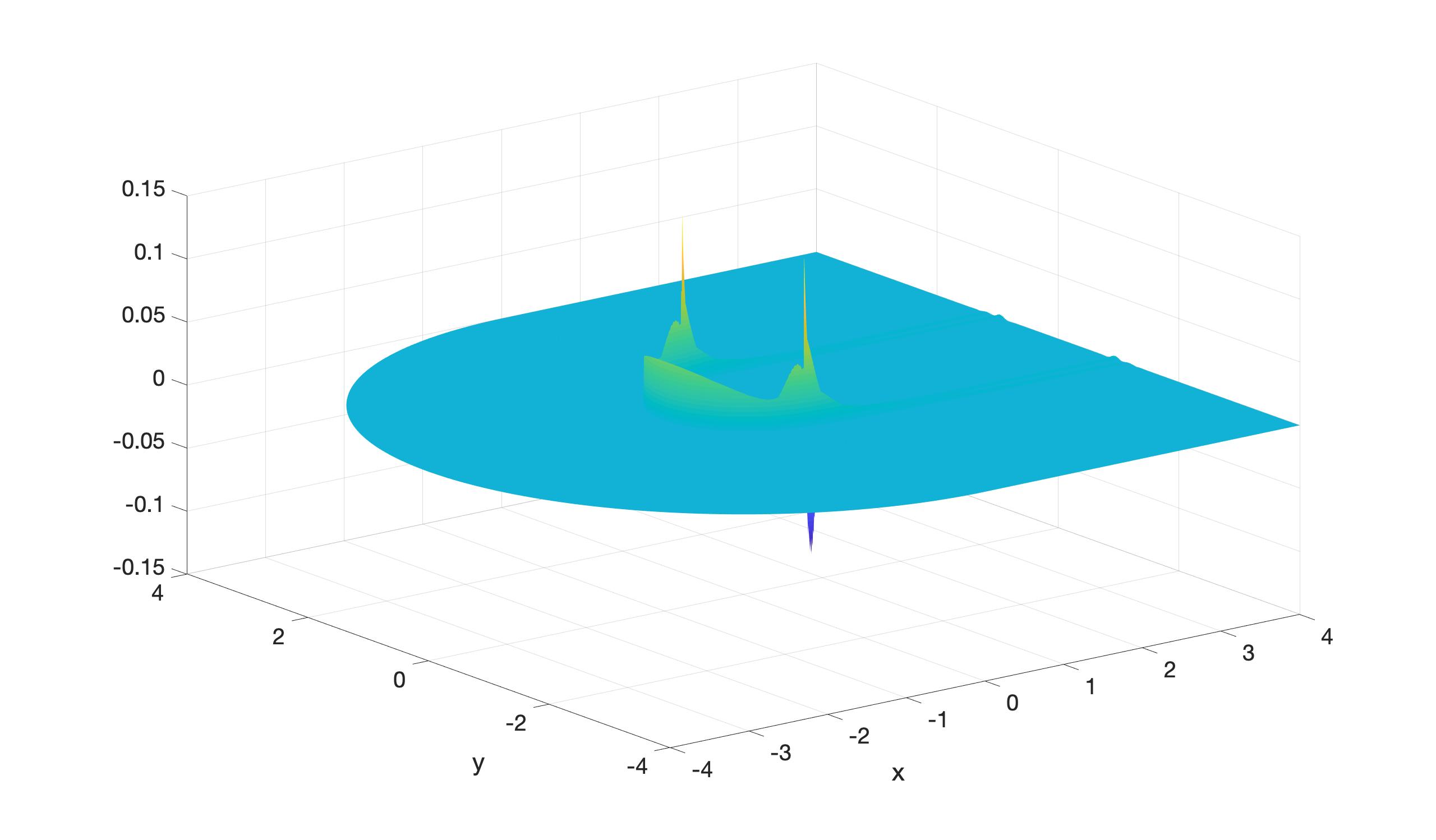}		
\caption{Approximate global error $ \bar U_1^{128}-\bar U_1^{2048}$ on the mesh $\tilde \Omega _A^N  \cup  \Omega ^N_S$, (\ref{fitted-mesh-special}) and (\ref{Shish-par-mesh})  with $ N=2048$, for $\ve =2^{-10}$}
	\label{Spike-N128}
\end{figure}
\begin{figure}		
\includegraphics[scale=0.18, angle=0]{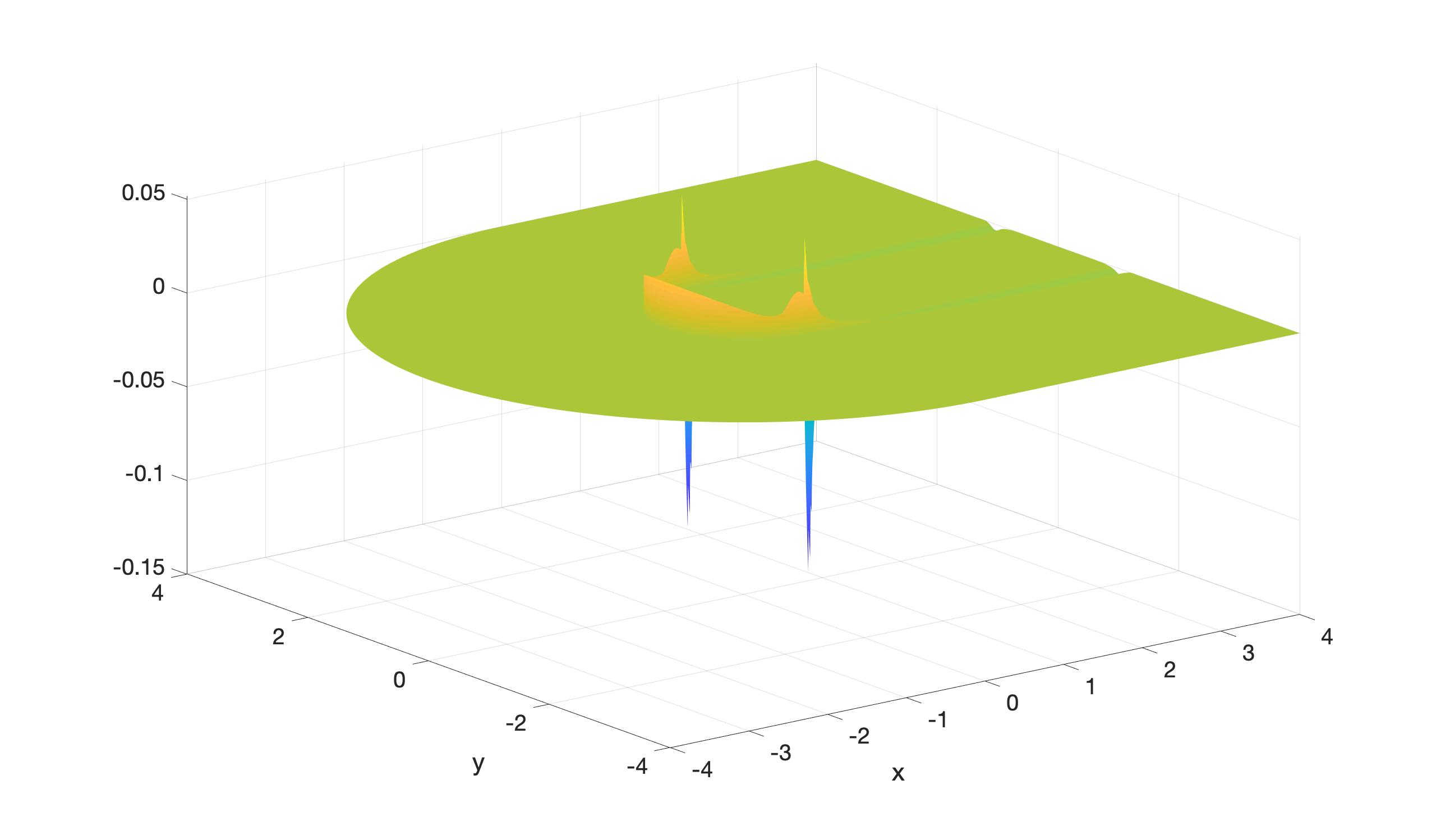}				
	\caption{Approximate global error $ \bar U_1^{256}-\bar U_1^{2048}$ on the mesh $\tilde \Omega _A^N \cup  \Omega ^N_S$, (\ref{fitted-mesh-special}) and (\ref{Shish-par-mesh})  with $ N=2048$, for $\ve =2^{-10}$}
	\label{Spike-N256}
 \end{figure}

Let us return to the bounds on the continuous solution given in Theorem \ref{main}. 
From  (\ref{center}), the bound on the solution remains constant along the parabolic path $y-1=-\frac{1}{2} x^2$. This is the motivation to introduce a third co-ordinate system  that is aligned to these parabolic curves.  Under this transformation, a mixed derivative term will appear in the transformed elliptic operator and we are required to restrict the dimensions of the sub-domain (where this new transformation is utilized) in order to preserve inverse-monotonicity of the corresponding discrete operator.

We introduce a patched region $Q:=Q^+\cup Q^-$, in a neighborhood of the vertical line $x=0$. 
We discuss the approach on the upper region 
\begin{equation}\label{defQ}
Q^+:= \{ (x,y) \vert y=t-x^2/2,\quad  0< x < L  <1,\ 1- \tau _3  < t < 1+3\delta \};
\end{equation}
with an analogous definition of the lower region $Q^-$.
The width $L$  and height $ 3\delta +\tau_3$ of this strip will be specified later in order to retain stability of the discrete operator. 

A natural coordinate system for this patched region $Q^+$ is
\[
s=x, \quad t= y+x^2/2, \qquad (y = t-s^2/2).
\]
Then let $\hat u(s,t) := u(x,y),  \hat Q^+ := (0,L) \times (1- \tau _3 ,1+3\delta )$ and 
\begin{eqnarray*}
u_x = \hat u _s +s \hat u_t, \ u_y= \hat u _t, \ u_{yy}=\hat u_{tt}, \quad
u_{xx} = \hat u_{ss} +2s \hat u_{ts} + s^2 \hat u_{tt} + \hat u_t.
\end{eqnarray*}
Under this transformation the problem (\ref{cont-prob}) on the patched region 
can be specified as follows
\begin{eqnarray*}
-\ve (\hat u_{ss} +2s \hat u_{ts} + (1+s^2) \hat u_{tt}) + \hat u _s + (s -\ve)\hat u _t = 0, (s,t) \in { \overline{ \hat Q^+}}\setminus   C_1;\\
  \hat u (s,t) =1, \quad \hbox{if} \quad s^2 + (t-s^2/2)^2 \leq 1.
\end{eqnarray*} 
 We will use the initial computed  approximation $U_1^N$ (\ref{initial-approx}) as boundary values to solve on the patched region $\hat Q^+$ as follows:
\begin{eqnarray*}
\hat L \hat u := -\ve (\hat u_{ss} +2s \hat u_{ts} + (1+s^2) \hat u_{tt}) +  \hat u _s + ( s -\ve)\hat u _t = 0,\quad  (s,t) \in \hat Q^+\setminus   C_1, \\
{ \hat u(s,1-\tau _3) =1}, \qquad  \hat u(s,1+3\delta) = \bar U_1(s,1+3\delta), \ s \in  (0,L), \\
\hat u(0,t) =\bar U_1(0,t),\qquad  \hat u_s (L,t) = 0,\ t \in  [1- \tau _3,1+3\delta ], \\
 \hat u(s,t) =1,\quad \hbox{if} \quad s^2 +(t-s^2/2)^2 \leq 1. 
\end{eqnarray*}
We  again note the use of a Neumann boundary condition   at the artificial internal boundary $s=L$.
To ensure that the corner point $(L, 1-\tau_3)$ of this patch $Q^+$ lies within the inner boundary ($ r <1$) of problem (\ref{cont-prob}), we require that
\[
-\frac{1}{2}L^2+1-\tau _3 < \sqrt{1-L^2} ;
\]
which, in turn,  requires that
\begin{equation}\label{constraint}
\tau _3   >  \frac{L^2}{2}\frac{1-\sqrt{1-L^2}}{1+\sqrt{1-L^2}}  =\frac{(1-\sqrt{1-L^2})^2}{2}.
\end{equation}

We now specify the next phase of our numerical algorithm, where we correct the initial approximation $\bar U_1$.

We numerically solve the problem on the patch $\hat Q^+$, in the $s,t$ coordinate system, using a Shishkin mesh $ \hat \Omega ^{N,M}_P$ with two transition points located at $t=1+\tau _3$ and $1 + \tau_3 +\tau _4$.
The choice for the transition parameters $\tau _3, \tau _4$  is motivated by the bounds  (\ref{center}) and (\ref{right}), written in the $(s,t)$ coordinates  
\[
\hat u (s,t) \leq C \min \{ e^{-\frac{t-1}{3\ve ^{2/3}}},  e^{-\frac{t-1}{\sqrt{\ve}}} \}, \quad \hbox{for} \quad s \in \ve ^{1/3} [-1,1].
\] 

{\bf The Shishkin mesh $\hat \Omega ^{N,M}_P$}
{\it
 We use a uniform mesh in the $s$-direction and a Shishkin mesh in the $t$-direction. 
The vertical strip $[1-\tau _3 ,1+3\delta]$  is split into the following four sub-regions
\[
 [1-\tau _3,1] \cup [1,1+\tau_3] \cup [1+\tau_3, 1+(\tau_3+\tau_4)] \cup [1+(\tau_3+\tau_4),1+3\delta]
\]
where
\begin{equation} \label{tau-3}
\tau _3 := \min \{ \delta , 3{\ve} ^{2/3} \ln M \}, \qquad \tau _4 := \min \{ \delta , 2\sqrt{\ve}\ln M \}, 
\end{equation}
and $M/4$ mesh elements are distributed uniformly within each of these sub-intervals.
}

Within this patched region, we adopt the following notation
\[
 h = N^{-1}L, \quad k \leq k_j:= t_j-t_{j-1} \leq K; \quad 2\bar k_j = k_j+k_{j+1}, \]
where, as $\tau _3 \leq \tau _4$, we have 
\[
 4M^{-1}  \tau_3  \leq k \leq K \leq   12 \delta M^{-1}.
\]
Note that if we assume that
\begin{equation} \label{M-and-N}
 12\delta \leq MN^{-1},
 \end{equation}
then the maximum mesh step $K$ in the vertical direction will be such that $K \leq N^{-1}$.

{\bf Stage 3}: {\it Numerical method for $U_C$ defined near the characteristic points}

Find $U_C$ such that:
\begin{subequations}\label{discrete-problem-transformed}
\begin{eqnarray} 
\hat L^{N,M}  \hat U_C:=
-\ve {\cal{L} }^{N,M}  \hat U_C +   D^{-}_s \hat U_C +(s_i-\ve) D^{\pm}_t \hat U_C   =0,\quad (s_i,t_j) \in  \hat \Omega ^{N,M}_P; \\
\hbox{where} \quad
  {\cal{L}} ^{N,M}  Y := \delta ^2_{ss} Y + 2s\delta _{st} Y +(1+s^2)\delta ^2_{tt} Y;\quad  2\delta _{st} := D^-_tD^-_s +D^+_tD^+_s 
	\label{patch-diff-op}
\end{eqnarray}
and for the remaining mesh points
\begin{eqnarray} 
\hat U_C(0,t_j)=\bar U_1(0,t_j), \quad D^-_s \hat U_C (L,t_j) = 0, \quad  1- \tau _3 \leq t_j \leq 1+3\delta ; \\
 \hat U_C(s_i,1-\tau _3)=1,\quad \hat U_C(s_i,1+ 3\delta ) = \bar U_1(s_i,1+ 3\delta ), \ 0 < s_i < L; \\
 \hat U_C(s_i,t_j)=1,\quad \hbox{if} \quad s_i^2 +(t_j-s_i^2/2)^2 \leq 1.
\end{eqnarray}
\end{subequations}
The presence of the mixed derivative term $\ve s \hat u _{st}$ in the transformed problem, creates the danger of loss of stability in the discretization of the differential operator $\hat L$ \cite{ray09}. However, by restricting the dimensions of this parabolic patch, we are able to preserve an appropriate sign pattern in the system matrix elements, so that the matrix ${\cal{L} }^{N,M}$ is an $M$-matrix. 

\begin{theorem} If  we choose the dimension  of the patch $\hat Q^+$ (\ref{defQ}) to satisfy 
\begin{equation}\label{def-L}
L \leq L_*:= 2\sqrt{NM^{-1}\tau _3} \quad \hbox{and} \quad 12 \delta \leq MN^{-1}
\end{equation}
then the finite difference operator $\hat L^{N,M}$ (\ref{discrete-problem-transformed}) satisfies a discrete comparison principle. That is, for any mesh function $Z$, 
\begin{eqnarray*}
\hbox{if} \quad Z(s_i,1-\tau _1) \geq 0, Z(s_i,1+3\delta) \geq 0, Z(0, t_j) \geq 0, \quad D^-_sZ(L, t_j) \geq 0, \ \forall (s_i,t_j) \in \bar Q^+ \\
\hbox{and} \quad \hat L^{N,M}Z(s_i, t_j) \geq 0, \qquad \forall (s_i,t_j) \in Q^+,
\end{eqnarray*}
then $Z(s_i, t_j) \geq 0, \quad \forall (s_i,t_j) \in \bar Q^+$.\end{theorem}
\begin{proof}
Let us  examine  the sign patterns of the second order operator ${\cal{L} }^{N,M}$ (defined in (\ref{patch-diff-op})).  From assumption (\ref{M-and-N}), we have that
\[
1-\frac{\bar k_j  s_i }{h} \geq 1- \frac{KL}{h} \geq 1- \frac{K}{2N} \geq 0.
\]
Using this, we see that for the internal mesh points
\[
-{\cal{L} }^{N,M} Y(s_i,t_j) = \sum _{n=j-1}^{j+1} \sum _{k=i-1}^{i+1} a_{kn} Y(s_k,t_n);
\]
where the sign of most of the coefficients $a_{kn}$  is  easily identified to be 
\begin{eqnarray*}
a_{i+1,j-1} =a_{i-1,j+1} =0;\quad
a_{i-1,j-1} = - \frac{s_i}{k_jh} \leq 0 &;&
a_{i+1,j+1} = -\frac{s_i}{k_{j+1}h} \leq 0; \\
a_{i,j} =\frac{2}{h^2} +\frac{2}{k_jk_{j+1}} \bigl(1+s_i^2 -  
 \frac{\bar k_j \vert s_i \vert)}{h} \bigr) &&>0; \\
a_{i,j-1} = - \frac{(1+s_i^2)}{\bar k_j k_j}+  \frac{s_i }{k_jh}=
 -\frac{1}{\bar k_jk_j} \bigl(1-\frac{\bar k_j  s_i }{h} +s_i^2\bigr) < 0 &;& a_{i,j+1} = -\frac{(1+s_i^2)}{\bar k_jk_{j+1}} +  \frac{ s_i }{hk_{j+1}} < 0.
\end{eqnarray*}
Finally, we look at the last two terms, 
\[
a_{i-1,j} =a_{i+1,j} =-\frac{1}{h^2} + \frac{s_i}{hk_j}, \quad a_{i+1,j} =-\frac{1}{h^2} + \frac{s_i}{hk_{j+1}}.
\]
Observe that
\[
\max \{a_{i-1,j}, a_{i+1,j} \} \leq  - \frac{1}{h^2} \bigl(1-\frac{Lh}{k}\bigr)
\leq- \frac{1}{h^2} \bigl(1-\frac{ML^2}{ 4N\tau _3}\bigr) \leq 0,
\]
if we choose $L$ such that (\ref{def-L}) is satisfied. This sign pattern on the matrix elements insures that the system matrix associated with the finite difference scheme is an M-matrix
\cite[pg.19]{fhmos}, which suffices to establish the result.  
\end{proof}

 The constraints in (\ref{def-L}) and (\ref{constraint}) are all satisfied if
\[ 
M=N, \ L^2 = 4 \tau _3 \quad \hbox{and} \quad \tau _3 \leq \delta \leq \frac{1}{12}, 
\]
as $2x > (1-\sqrt{1-4x})^2$ for $0 < x  < \frac{2}{9}$. 

In the final  phase, we solve the following discrete problem over the rectangle 
\begin{equation}\label{downwind}
 S^*:= (L_*,R) \times (-R, R) \subset S,\qquad L_* :=2 \sqrt {NM^{-1} \tau _3}; 
\end{equation}
 using the mesh $ \Omega ^N_S$, which was  defined in (\ref{Shish-par-mesh}). 

{\bf Stage 4}: {\it Numerical method for $U_D$ defined over the downwind region $S^*$}

 Find  $U_D$  such that 
\begin{subequations}\label{final-scheme}
\begin{eqnarray} 
L^{N,M}U_D(x_i,y_j)     = 0 ,  \ (x_i,y_j) \in  \Omega ^N_S \setminus \bar C_1; 
\\
 U_D (L_*,y_j) =\left \{ \begin{array}{ll} \bar U_C (L_*,y_j), \  &y_j \in (-1-3\delta, 1+3\delta); \\
 \bar U_1 (L_*,y_j),\ &y_j \in [-R,-1-3\delta] \cup  [1+3\delta, R],\ \delta  \leq M/(12N);
\end{array}\right. \\
D^-_x  U_D (R,y_j)=0,  \ -R < y_j < R,\quad  U_D(x_i,-R)= U_D(x_i,R)=0, \ x_i \in [L_*,R] . 
\end{eqnarray}\end{subequations}

Then our corrected numerical approximation is given by
\begin{equation}\label{corrected-approx}
\bar U_2^N (x,y) := \left \{
\begin{array}{lll}
\bar U_1(x,y), & \hbox{for} \quad (x,y) \in \bar \Omega  \setminus (\{ x  \geq L^* \} \cup Q^+\cup Q^-) \\
\bar U_C(x,y) , & \hbox{for} \quad (x,y) \in  Q^+\cup Q^-, \\
\bar U_D(x,y) , & \hbox{for} \quad (x,y) \in (\{ x  \geq L^* \} \cap \Omega )\setminus (Q^+\cup Q^-)
\end{array}\right. .
\end{equation}
In the next section, we present some numerical results to illustrate the convergence properties of this corrected approximation, which is defined  across  three different coordinate systems. A schematic image of the composite mesh 
 $\tilde \Omega _A^{N,M} \cup \hat \Omega _P^{N,M}\cup \Omega _S^{N,M}$ is presented in Figure \ref{Fig.5}.
\begin{figure} 
\center{\includegraphics[scale=0.15, angle=0]{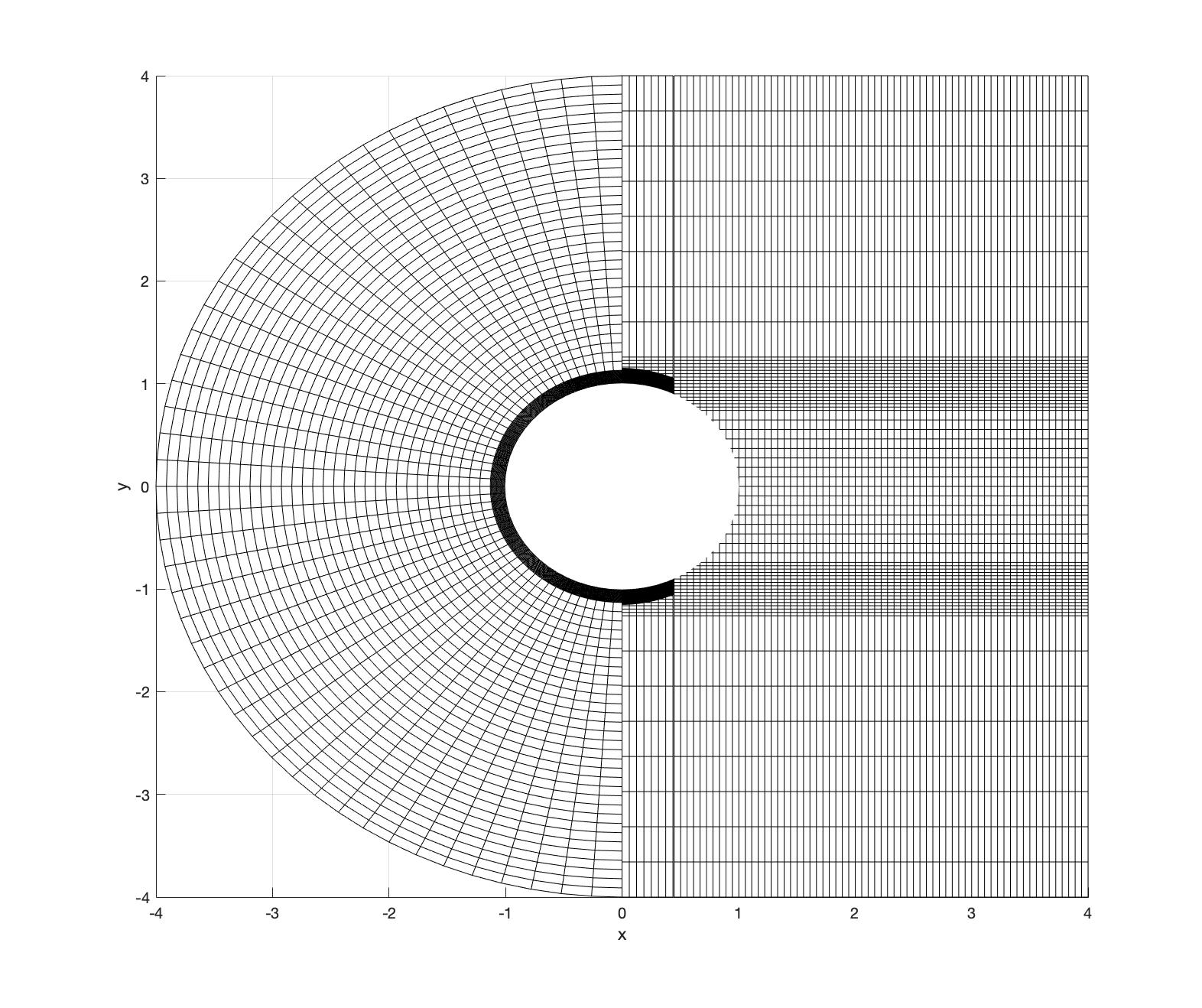}}
\caption{A schematic image of the composite mesh  $\tilde \Omega _A^{N,M} \cup \hat \Omega _P^{N,M}\cup \Omega _S^{N,M}$}
\label{Fig.5}
\end{figure}

\section{Numerical results}
 
In the previous section, we have seen that the initial approximations $\bar U_1^N$ displayed a lack of convergence, due to the presence of  large errors in the neighbourhood of the characteristic points. The corrected approximations $\bar U_2^N$ incorporate a parabolic patch near these points.  In the numerical experiments in this section, we have taken $M=N,\ \delta =0.05$ and, for ease of generating Tables, we have simply taken
$L=  2 \sqrt{\min \{ \delta ,  \ve^{2/3}  \ln 2048 \}} $ in the patched region. 
 When we include the patch,  we  observe convergence  in Table \ref{Tab3} of the corrected approximations over this parabolic patch $\Omega _P^N$ over an extensive range of $\ve $ and $N$. 
In Table \ref{Tab4},  the global orders of convergence over the entire domain for the corrected approximation $\bar U_2^N$ are  given.  These orders indicate that the corrected approximations are converging for all values of $\ve \in [2^{-20},1]$. In the final Table \ref{Tab5}, the approximate global errors over the entire domain are displayed for all $\ve \in [2^{-20},1]$.
We observe that as $\ve \rightarrow 0$ the global errors continue to grow for each fixed $N$. Hence, the method appears  not to be parameter-uniform. Nevertheless, for any fixed value of $\ve$ we do observe convergence as $N $ increases. In particular, 
we see in Figures \ref{Fig.8} and  \ref{Fig.9}  that for the corrected approximations  $\bar U_2^N$, the approximate global errors $ \bar U_2 -\bar U_2^{2048}$ essentially halve as the number of mesh points are doubled. This in sharp contrast to the approximate global errors $\bar U_1 -U_1^{2048}$ displayed in  Figures \ref{Spike-N128}, \ref{Spike-N256}.

 \begin{figure}[h!] \centering
		\includegraphics[scale=0.18, angle=0]{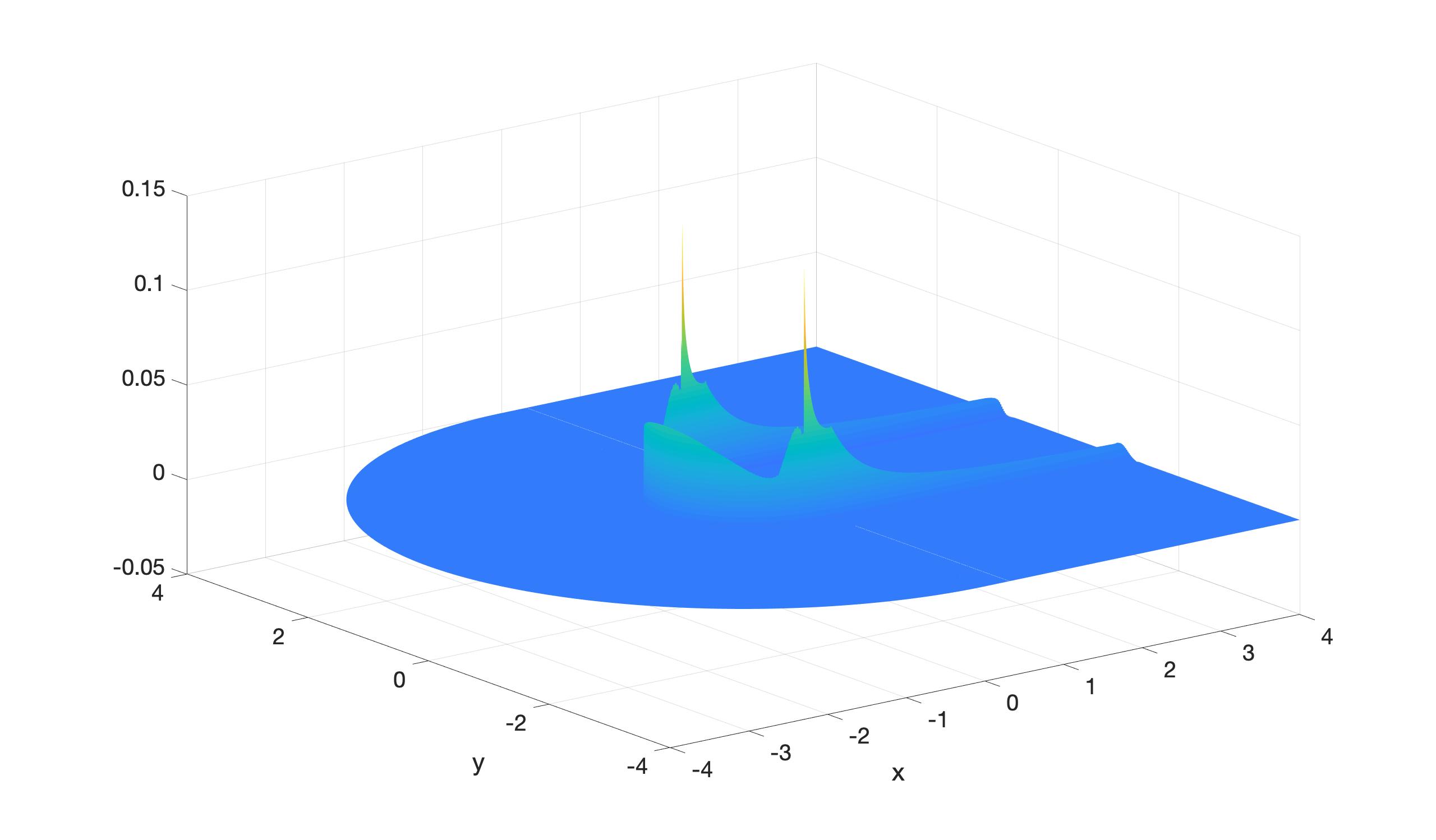}		
\caption{Approximate global error $ \bar U_2^{128}-\bar U_2^{2048}$  in corrected approximation  for $\ve =2^{-10}$ }
	\label{Fig.8} 
\end{figure}
\begin{figure}
\includegraphics[scale=0.18, angle=0]{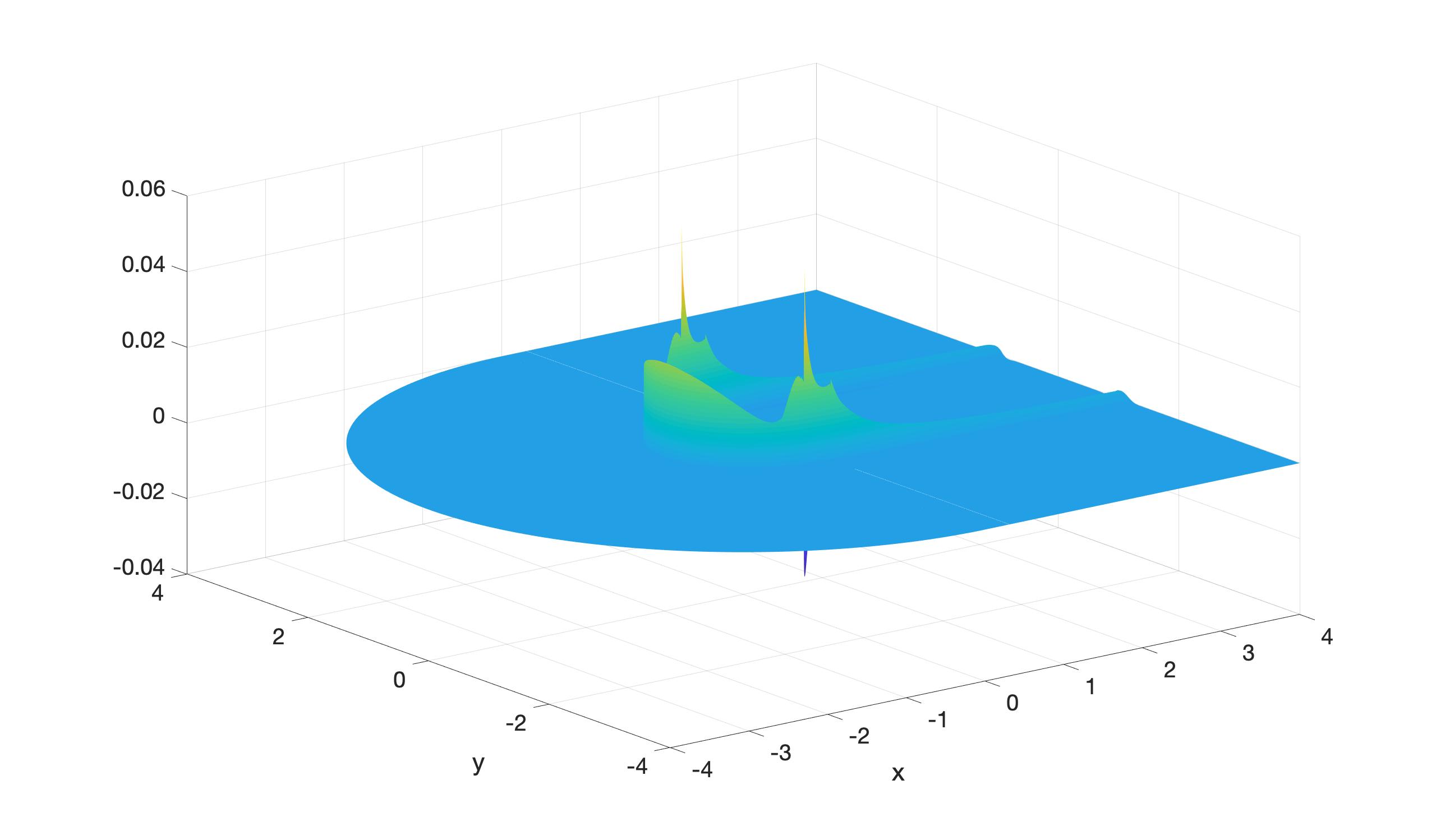}			
\caption{Approximate global error $ \bar U_2^{256}-\bar U_2^{2048}$  in corrected approximation for $\ve =2^{-10}$   }
	\label{Fig.9} 
 \end{figure}

\begin{table}[ht]
\centering\small
\begin{tabular}{|r|rrrrrrr|}\hline 
$\varepsilon| N$&$N=8$& 16&32&64&128&256&512\\ \hline 
 $1$ &   1.4739 &   0.3521 &   0.6916  &   1.1128  &   0.7936 &   0.9664 &   0.9716\\ 
 $2^{-2}$ &   1.7825 &   0.2514 &   0.4127  &   1.1106  &   0.6181 &   0.9426 &   0.9508\\ 
 $2^{-4}$&   1.4364 &   0.6690 &   0.4384  &   0.8073  &   0.6223 &   0.8619 &   0.4013\\ 
 $2^{-6}$ &   0.2310 &   0.7570 &   1.5302  &  -0.6245  &   2.0571 &   0.5677 &   0.2180\\ 
 $2^{-8}$&  -0.2532 &   0.3617 &   0.8478  &  -0.0286  &   1.1685 &   1.2720 &   0.1293\\ 
 $2^{-10}$ &  -0.0002 &   0.0086 &   0.4198  &   1.0663  &   1.6521 &   1.2344 &   0.8949\\ 
 $2^{-12}$&  -0.0208 &  -0.0051 &   0.1917  &   0.8547  &   1.4528 &   1.2501 &   1.0148\\ 
 $2^{-14}$&   0.1888 &  -0.0509 &   0.1504  &   0.5917  &   1.2201 &   1.3731 &   1.0606\\ 
 $2^{-16}$ &   0.3670 &  -0.0685 &   0.0477  &   0.3999  &   1.0007 &   1.6196 &   1.1642\\ 
 $2^{-18}$ &   0.3904 &  -0.0179 &  -0.0203  &   0.2101  &   0.7656 &   1.4258 &   1.3736\\ 
 $2^{-20}$ &   0.7398 &  -0.2868 &  -0.0467  &   0.0795  &   0.5156 &   1.2055 &   1.6702\\ \hline 
$\bar p^N_{R^{20}_\ve}$  &   0.3593 &   0.0521 &  -0.0351  &   0.1095 &   0.5156 &   1.2055  &   1.6702\\ \hline
\end{tabular}
\caption{ Computed double-mesh global orders of convergence $\bar p^{N}_{\epsilon}$ for the corrected approximations $\bar U_C$ (\ref{discrete-problem-transformed}) measured over the patched region  $Q$ (\ref{defQ}), where $L= 2 \sqrt{\min \{ \delta ,  \ve^{2/3}  \ln 2048 \}} $ and  $\delta =0.05$.}
\label{Tab3}
\normalsize
\end{table}

\begin{table}[ht]
\centering\small
\begin{tabular}{|r|rrrrrrr|}\hline 
$\varepsilon| N$&$N=8$& 16&32&64&128&256&512\\ \hline 
$1$ &   2.0261 &   0.6227 &   0.2352  &   0.7464  &   0.8763 &   0.9448 &   0.9799\\ 
$2^{-2}$ &   2.0506 &   1.0501 &   0.3720  &   0.3373  &   0.8437 &   0.9907 &   0.4472\\ 
$2^{-4}$  &   1.9773 &   0.1282 &   1.9129  &   0.8324  &   0.3435 &   0.3734 &  -0.3310\\ 
$2^{-6}$ &   0.2310 &   0.7570 &   1.5302  &  -0.6245  &   2.5573 &   0.2520 &   0.0335\\ 
$2^{-8}$ &  -0.2532 &   0.3617 &   0.8478  &  -0.0286  &   1.1685 &   1.2720 &   0.1293\\ 
$2^{-10}$ &  -0.0002 &   0.0086 &   0.4198  &   1.0663  &   1.6521 &   1.2344 &   0.8949\\ 
$2^{-12}$  &  -0.0208 &  -0.0051 &   0.1917  &   0.8547  &   1.4528 &   1.2501 &   1.0148\\ 
$2^{-14}$  &   0.1888 &  -0.0509 &   0.1504  &   0.5917  &   1.2201 &   1.3731 &   1.0606\\ 
$2^{-16}$  &   0.3670 &  -0.0685 &   0.0477  &   0.3999  &   1.0007 &   1.6196 &   1.1642\\ 
$2^{-18}$  &   0.3904 &  -0.0179 &  -0.0203  &   0.2101  &   0.7656 &   1.4258 &   1.3736\\ 
 $2^{-20}$ &   0.7398 &  -0.2868 &  -0.0467  &   0.0795  &   0.5156 &   1.2055 &   1.6702\\ \hline 
$\bar p^N_{R^{20}_\ve}$&   0.3593 &   0.0521 &  -0.0351  &   0.1095  &   0.5156 &   1.2055  &   1.6702\\ 
\hline
\end{tabular}
\caption{Computed double-mesh  orders of global convergence $\bar p^{N}_{\epsilon}$ for the corrected approximations $\bar U_2$ (\ref{corrected-approx}) measured over the entire domain $\Omega$}
\label{Tab4}
\normalsize
\end{table}

\begin{table}[ht]
\centering\small
\begin{tabular}{|r|rrrrrrr|}\hline 
$\varepsilon| N$&$N=8$& 16&32&64&128&256&512\\
\hline 
$1$ &   0.1689 &   0.0828 &   0.0686  &   0.0435  &   0.0245 &   0.0127 &   0.0057\\
 $2^{-2}$ &   0.2735 &   0.1434 &   0.0765  &   0.0530  &   0.0311 &   0.0197 &   0.0100 \\
 $2^{-4}$ &   0.4079 &   0.2272 &   0.1415  &   0.0887  &   0.0515 &   0.0305 &   0.0169 \\
 $2^{-6}$&   0.5720 &   0.4206 &   0.2369  &   0.1186  &   0.1109 &   0.0330 &   0.0188 \\
 $2^{-8}$ &   0.6790 &   0.5657 &   0.4065  &   0.2034  &   0.1852 &   0.0650 &   0.0123 \\
 $2^{-10}$&   0.7331 &   0.6325 &   0.4882  &   0.3030  &   0.1344 &   0.0518 &   0.0198 \\
$2^{-12}$ &   0.7769 &   0.6916 &   0.5624  &   0.3859  &   0.1971 &   0.0698 &   0.0326 \\
$2^{-14}$ &   0.8140 &   0.7414 &   0.6287  &   0.4662  &   0.2718 &   0.1100 &   0.0446\\
$2^{-16}$ &   0.8458 &   0.7847 &   0.6871  &   0.5419  &   0.3462 &   0.1615 &   0.0547 \\
$2^{-18}$ &   0.8718 &   0.8213 &   0.7393  &   0.6122  &   0.4329 &   0.2280 &   0.0799 \\
$2^{-20}$ &   0.8950 &   0.8519 &   0.7830  &   0.6720  &   0.5095 &   0.3025 &   0.1213 \\
\hline 
\end{tabular}
\caption{Approximate global errors $\Vert \bar U^N_2 - \bar U^{2048}_2 \Vert _\Omega$ over the  domain $\Omega$}
\label{Tab5}
\normalsize
\end{table}
The computed orders of convergence in Table \ref{Tab4} suggest that this multi-stage numerical 
method is producing a converging sequence of numerical approximations to the analytical solution of the Hemker problem (\ref{cont-prob}) across the entire domain for singular perturbation values $\ve \in [2^{-20},1]$.

\section{Conclusions}

Based on parameter explicit pointwise bounds on how the continuous solutions decays away from the circle, a numerical method was constructed for the Hemker problem. There are no spurious oscillations present in the numerical solutions, as we use simple upwinding in all co-ordinate directions used. Several layer adapted Shishkin meshes are utilized and these grids are aligned both to the geometry of the domain and to the dominant direction of decay within the boundary/interior  layer functions. Numerical experiments indicate that the method is producing accurate approximations over an extensive range of the singular perturbation parameter.  Hence,  the method is stable for all values of the singular perturbation parameter and the numerical approximations are converging to the continuous solution for each value of the parameter; however,  this convergence is not uniform in the singular perturbation parameter.   
\vskip1.0cm


{\bf Appendix 1: Bounds on the continuous solution}

\begin{enumerate}

\item 

For  $0< \kappa \leq 1$, consider the  barrier function
\begin{equation}\label{barrier}
 B^-(x,y) := \Bigl\{ \begin{array}{cc} 
 e^{\frac{ \kappa \cos (\theta ) (r-1)}{\ve}} &,\ \cos (\theta )  \leq 0,\quad r \geq 1, \\
1&,\ x  \geq 0,\quad x^2+y^2 \geq 1
\end{array}.
\end{equation}
For the subregion where $x<0$, note the following expressions for the partial derivatives of this barrier function:
\begin{eqnarray*}
\tilde B^-_\theta &=& -\frac{ \kappa \sin (\theta ) }{\ve} (r-1)\tilde B^- ; \\
\tilde B^-_{\theta \theta} 
&=& -\bigl(  \frac{\kappa \cos (\theta) }{\ve} (r-1) - \frac{ \kappa ^2 \sin ^2 (\theta) }{\ve ^2} (r-1)^2 \bigr)\tilde B^- ; \\
\tilde B^-_r &=& \frac{\kappa \cos (\theta) }{\ve} \tilde B^-,   \qquad
\tilde B^-_{rr} = \frac{\kappa ^2 \cos ^2 (\theta) }{\ve ^2}\tilde  B^-.
\end{eqnarray*}
Combining these expressions, we can deduce  that  
\begin{eqnarray*}
\tilde L \tilde B^- \geq \frac{\kappa }{\ve} \sin ^2 (\theta) (1-\kappa +\frac{\kappa}{r})  (1-\frac{1}{r}) \tilde B^- \geq 0,\ \cos (\theta )  <0,\ r > 1;\\
{ \ \bigl[ \frac{\partial  B^-}{\partial n} \bigr] _{x=0} = B^-_x (0^-,y)  =  \frac{\kappa }{\ve}  (1-\frac{1}{\vert y \vert}) \geq 0;}\ 
B^-(x,y) = 1, \ x^2+y^2 =1;  \frac{\partial B^-}{\partial x}(R,y) = 0 .
\end{eqnarray*}
 Using Theorem \ref{compar3}  we establish the bound (\ref{left}). 

\item 
 Consider the following barrier function
\[
 B^+(x,y) := \Bigl\{ \begin{array}{cc} e^{C_1(1+x)} e^{-\frac{ (y-1)}{\sqrt{\ve}} },&  \ y \geq 1; \quad  C_1 \geq 2\\
e^{C_1(1+x)} &\ ,\quad y \leq 1
\end{array}.
\]
In the region where $y>1$,  note the following expressions for the partial derivatives of this function:
\[
 B^+_x = C_1B^+\quad \hbox{and} \quad  -\ve  B^+_{yy} = -  B^+
 \]
and so, for $\ve$ sufficiently small, 
\[
L  B^+ 
=  \bigl(C_1- \ve C_1^2 - 1 \bigr)B^+ \geq 0, \quad \forall  y >1. 
\]
  Note also that
\begin{eqnarray*}
B^+(x,y) \geq 1, \ \quad \hbox{if} \quad x^2+y^2=1; \
{ \ \bigl[ \frac{\partial  B^+}{\partial n} \bigr] _{y=1} = -B_y^+(x,1^+) >0} \quad \hbox{and} \quad B^+_x(R,y) \geq 0.
\end{eqnarray*}
 Using Theorem \ref{compar3} we  establish the bound (\ref{right}) for $y>0 $ and  using symmetry  we deal with the case of $y<0$. 

\item For $y >0$, consider the following  function, defined in a neighbourhood of the line $x=0$, 
\[
B(x,y) := \bigl(1+\frac{\alpha x}{\ve ^{1/3}} \bigr)(1-\alpha C)^{-1}e^{-\frac{\kappa (0.5x^2+y-1)}{\ve ^{2/3}}},\quad -C\ve ^{1/3} < x < C\ve ^{1/3}, \quad \alpha C < 1;
\]
where the possible ranges for the positive parameters $\alpha, \kappa, C$ will be specified below.
Note the following expressions for the partial derivatives of this function:
\begin{eqnarray*}
 B_x &=& \frac{(1-\alpha C)^{-1}}{\ve ^{1/3}}\bigl( \alpha - \frac{\kappa x}{\ve ^{{1/3}}} (1+\frac{\alpha x}{\ve ^{1/3}}) \bigr)e^{-\frac{ \kappa (0.5x^2+y-1)}{\ve ^{{2/3}}}}; \\
\ve  B_{yy} &=& \frac{\kappa ^2(1-\alpha C)^{-1}}{\ve ^{1/3}}\bigl(1+\frac{\alpha x}{\ve ^{1/3}}\bigr)e^{-\frac{ \kappa (0.5x^2+y-1)}{\ve ^{{2/3}}}}; \\
\tilde B_{xx} 
&=&  \frac{1}{\ve } O(\ve ^{1/3})e^{-\frac{ \kappa (0.5x^2+y-1)}{\ve ^{{2/3}}}},\quad -C\ve ^{1/3} < x < C\ve ^{1/3}.
\end{eqnarray*}
Let us introduce $\xi := x \ve ^{-1/3}$, then for $\xi \in C (-1,1)$, $0< \alpha C < 1$,
\begin{eqnarray*}
-\ve  B_{yy} + B_x &=&\frac{(1-\alpha C)^{-1}}{\ve ^{1/3}}\bigl(\alpha - \kappa \xi (1+\alpha \xi) - \kappa ^2 (1+\alpha \xi )\bigr)e^{-\frac{ \kappa (0.5x^2+y-1)}{\ve ^{{2/3}}}} \\
&\geq &\frac{1}{\ve ^{1/3}}\bigl(\frac{\alpha}{(1+\alpha C)}  - C\kappa - \kappa ^2 \bigr)e^{-\frac{ \kappa (0.5x^2+y-1)}{\ve ^{{2/3}}}}.
\end{eqnarray*}
In order that $ LB(x,y) \geq 0, x \in C\ve ^{1/3} (-1,1)$  (and recalling that $\alpha C < 1$) we  impose the following constraints:
\begin{eqnarray}\label{constraints}
 \kappa  \leq  \frac{C}{3}, \quad \alpha >\frac{8}{9} C^2, \quad  C^3 < \frac{9}{8}.
\end{eqnarray}
For simplicity, we  take the particular values, 
\[
C =1, \ \alpha =0.9  \quad \hbox{and} \quad  \kappa =\frac{1}{3} . 
\]
 From this,
\[
B_x (x,y) >0, \quad B_y (x,y) <0; \quad -\ve ^{1/3} \leq  x \leq  \ve ^{1/3}, \ y >0.
\]
By restricting the domain of the function $B(x,y)$ to the strip
\[
S:= \{ (x,y) | \vert x \vert < \ve ^{1/3}, -\mu  \ve ^{2/3} < 0.5x^2+y -1 < \mu  \ve ^{2/3}.\}
\]
then $e^{\kappa \mu }\geq B(x,y) \geq e^{-\kappa \mu },\ (x,y) \in \bar S$ and
\[
e^{\kappa \mu } B(x,y) \geq 1 \geq u(x,y) ,\qquad  (x,y) \in \bar S \setminus S.
\]
 Using Theorem \ref{compar1} we establish the bound (\ref{center}) for $y>0 $ and using symmetry we deal with the case of $y<0$.

\end{enumerate}

\end{document}